\documentclass[11pt]{article} 

\usepackage[utf8]{inputenc} 

\usepackage[margin=1in]{geometry} 
\usepackage{graphicx} 

\usepackage{booktabs} 
\usepackage{array} 
\usepackage{paralist} 
\usepackage{verbatim} 
\usepackage{subfig} 
\usepackage{mathrsfs}
\usepackage{amssymb}
\usepackage{amsthm}
\usepackage{amsmath,amsfonts,amssymb,esint}
\usepackage{graphics}
\usepackage{enumerate}
\usepackage{mathtools}
\usepackage{xfrac}
\usepackage{bbm}
\usepackage{fancyhdr} 
\usepackage[nottoc,notlof,notlot]{tocbibind} 
\usepackage[titles,subfigure]{tocloft} 

\def\theequation{\thesection.\arabic{equation}}
\numberwithin{equation}{section}

\newtheorem{Thm}{Theorem}

\newtheorem{theorem}{Theorem}[section]

\newtheorem{Prop}[theorem]{Proposition}
\newtheorem{Def}[theorem]{Definition}
\newtheorem{Lem}[theorem]{Lemma}
\theoremstyle{definition}

\newtheorem*{Rem}{Remark}

\newcommand{\eps}{\varepsilon}

\renewcommand*{\tilde}{\widetilde}
\renewcommand*{\hat}{\widehat}

\newcommand{\mb}{\mathbb}
\newcommand{\mc}{\mathcal}
\newcommand{\ol}{\overline}
\newcommand{\onabla}{\overline{\nabla}}
\newcommand{\oDelta}{\overline{\Delta}}
\newcommand{\mean}[1]{\,-\hskip-1.08em\int_{#1}}
\newcommand{\textmean}[1]{- \hskip-.9em \int_{#1}}

\usepackage[usenames,dvipsnames]{color}
\usepackage[colorlinks=true, pdfstartview=FitV, linkcolor=blue, citecolor=blue, urlcolor=blue]{hyperref}

\title{Partial regularity for the steady hyperdissipative fractional Navier--Stokes equations}
\author{Eric Chen\thanks{Department of Mathematics, Princeton University.
{\footnotesize \href{mailto:ecchen@math.princeton.edu}{ecchen@math.princeton.edu}.}
}
}

\begin{document}

\maketitle

\begin{abstract}
We extend the Caffarelli--Kohn--Nirenberg type partial regularity theory for the steady $5$-dimensional fractional Navier--Stokes equations with external force to the hyperdissipative setting. In our argument we use the methods of Colombo--De Lellis--Massaccesi to apply a blowup procedure adapted from work of Ladyzhenskaya--Seregin.
\end{abstract}


\section{Introduction}
In this paper we study the partial regularity of suitable weak solutions of the steady fractional Navier--Stokes equations in $\mb{R}^5$. In general, on $\mb{R}^n$ these equations are
\begin{align}
	\begin{cases}
		(-\Delta)^s u+(u\cdot\nabla) u+\nabla p=f,
		\\
		\text{div }u=0.\label{NS5}
	\end{cases}
\end{align}
Here $u:\mb{R}^n\rightarrow\mb{R}^n$ is a velocity field, $p:\mb{R}^n\rightarrow\mb{R}$ is the pressure, and $f:\mb{R}^n\rightarrow\mb{R}^n$ is a divergence-free external force, while $s\geq 0$ gives the power of the fractional Laplacian with Fourier symbol $|\xi|^{2s}$. We will be concerned with the case $s>1$ when $n=5$.

The equations of \eqref{NS5} are exactly the equations satisfied by the time-independent (steady) solutions of the fractional Navier--Stokes equations in $\mb{R}^n$, which are
\begin{align}
    \begin{cases}
    \partial_t u+(-\Delta)^s u+(u\cdot\nabla) u+\nabla p=f,
    \\
    \text{div }u=0.\label{NS5time}
    \end{cases}
\end{align}
When $s=1$ and $n=3$ these become the classical Navier--Stokes equations. These have been studied as physical models also for fractional powers $s$ \cite{Mercado}.

\subsection{Previous results}

In the classical setting $s=1$ and $n=3$, the foundational result of Caffarelli--Kohn--Nirenberg \cite{CKN} constructed suitable weak solutions of \eqref{NS5time} and established that $\mc{H}^1(\mc{S})=0$ for $f\in L^q_{t,x}$ for any $q>\frac{5}{2}$, where
\begin{align}
\mc{S}=\{(x,t):\ u\text{ is not locally bounded at }(x,t)\}.\notag
\end{align}
Here $\mc{S}$ is called the singular set of $u$ and points not in $\mc{S}$ are called regular points. F. Lin \cite{L} later gave a different proof of this result when $f=0$ via a blow-up argument which was expanded upon and extended by Ladyzhenskaya--Seregin \cite{LS} to situations when $f$ lies in a parabolic Morrey space---in particular their result allows for $f\in L^q_{t,x}$ for all $q>\frac{5}{2}$ as in \cite{CKN}. Later I. Kukavica \cite{Ku} weakened the hypothesis on the force term to $f\in L^2_{t,x}$ by simplifying the original proof of Caffarelli--Kohn--Nirenberg by using a Morrey-type estimate due to M. O'Leary \cite{OL}. By now there are also other proofs of the Caffarelli--Kohn--Nirenberg theorem, such as \cite{Vasseur}. 

In the case $s<1$ and $n=3$, Tang--Yu \cite{TY} proved that suitable weak solutions of \eqref{NS5time} satisfy $\mc{H}^{5-4s}(\mc{S})=0$ for $3/4<s<1$. Later Chen--Wei \cite{CW} extended this result to the boundary case $s=3/4$ below which the Sobolev embedding no longer gives the necessary control of terms in energy inequality satisfied by suitable weak solutions. In the case $s>1$ and $n=3$, Colombo--De Lellis--Massaccesi \cite{CDM} have recently shown that $\mc{H}^{5-4s}(\mc{S})=0$ also holds for $1<s\leq\frac{5}{4}$. This improves an earlier result of Katz--Pavlovi\'{c} \cite{KP} in this range, which showed that the Hausdorff dimension of the singular set in space at the time of first blow-up is at most $5-4s$, and connects the result of Caffarelli--Kohn--Nirenberg to the global well-posedness of \eqref{NS5time} when $s\geq\frac{5}{4}$ \cite{JL}. 

The assumption of time-independence leads to better Hausdorff measure estimates than those for the non-steady equation \eqref{NS5time}. Partial regularity for solutions in $\mb{R}^5$ of the steady (time-independent) Navier--Stokes equations \eqref{NS5} was first studied by M. Struwe \cite{Struwe} when $s=1$. The dimension $n=5$ for the steady equations is natural to consider because in the classical setting $s=1$ and $n=3$ of \eqref{NS5time}, time corresponds to $2$ space dimensions (see the dimension table of \cite[p. 774]{CKN}). Struwe showed that $\mc{H}^{1}(\mc{S})=0$. Gerhardt \cite{Gerhardt} proved the corresponding result for the steady equations in dimension $4$ when $s=1$, showing that $\mc{S}=\emptyset$.

In the case $s<1$ of \eqref{NS5}, Tang--Yu \cite{TY2} proved in dimension $3$ that $\mc{H}^{5-6s}(\mc{S})=0$ for $s\in(1/2,5/6)$ and $\mc{S}=\emptyset$ when $s\in[5/6,1]$. Later Guo--Men \cite{GM} generalized this result to corresponding partial regularity statements in the dimensions $4$ and $5$.

\begin{Rem}
We will discuss here the reason that we consider $n=5$ in \eqref{NS5}. The partial regularity results starting from \cite{CKN} rely crucially on use of the local energy inequality satisfied by suitable weak solutions. In particular we can see from the energy inequality \eqref{ssen} that for \eqref{NS5} in dimension $n$ it is important to control the $L^3(\Omega)$ norm of $u$ by its $H^s(\Omega)$ norm in bounded domains $\Omega\subseteq\mb{R}^n$. The Sobolev embedding then requires that $6s>n$, restricting the study the partial regularity theory of \eqref{NS5} to low dimensions for $s\in(1,2)$. Since we already have full regularity ($\mc{S}=\emptyset$) from \cite{TY2} and \cite{Gerhardt} in dimensions $3$ and $4$ when $s\nearrow 1$ it is therefore natural to study partial regularity for \eqref{NS5} when $s>1$ in the next lowest dimension $n=5$. This would also be a natural extension of the partial regularity results when $n=5$ of \cite{Struwe} and \cite{GM}, which together cover $s\leq 1$. In $6$ dimensions, partial regularity for \eqref{NS5} has also been obtained by Dong--Strain \cite{DongStrain}, and the arguments presented below should also work with the appropriate modifications for \eqref{NS5} in $\mb{R}^6$ with $s\in(1,\frac{4}{3}]$.
\end{Rem}

\subsection{Main results}

Our main theorem is the following partial regularity result for suitable weak solutions of \eqref{NS5}:

\begin{Thm}\label{supermain}
    Let $(u,p)$ be a suitable weak solution in $\mb{R}^5$ of \eqref{NS5} as in Definition \ref{suitable} with $s\in(1,2)$ and force $f\in  L^q(\mb{R}^5)$ for some $q>\frac{5}{2s}$. Then $\mc{H}^{7-6s}(\mc{S})=0$ if $s<\frac{7}{6}$, and $u$ is everywhere regular (in fact H\"{o}lder continuous) if $s\geq\frac{7}{6}$. Here $\mc{H}^\beta$ is the Hausdorff measure in $\mb{R}^5$ defined by
    \begin{align}
        \mc{H}^\beta(E)=\lim_{\delta\rightarrow 0}\inf\left\{\sum_i r_i^\beta:\ E\subset\bigcup_i B_{r_i}(x_i,t_i)\text{ and }r_i<\delta\ \forall i\right\},\quad\text{for }E\subseteq\mb{R}^5.\notag
    \end{align}
\end{Thm}

Theorem \ref{supermain} follows by a standard covering argument (included in the Appendix, see also \cite[Section 6]{CKN}) from the following $\varepsilon$-regularity theorem:

\begin{Thm}\label{epsmain}
    There exists $\varepsilon>0$ such that if $(u,p)$ is a suitable weak solution in $\mb{R}^5$ of \eqref{NS5} as in Definition \ref{suitable} with $s\in(1,2)$ and force $f\in L^q(\mb{R}^5)$ for some $q>\frac{5}{2s}$ , and
    \begin{align}
        \limsup_{r\rightarrow 0} r^{-7+6s}\int_{B_r^+(x)}y^b|\onabla(\nabla u)^\flat|^2\ dx\ dy<\varepsilon,
    \end{align}
    then $u$ is regular at $x\in\mb{R}^5$.
\end{Thm}

Here $(\nabla u)^\flat\in H^1(\mb{R}^{n+1}_+,y^b)$ is an extension of $\nabla u$ on $\mb{R}^n$ to the upper half-plane $\mb{R}^{n+1}_+$, as will be clarified in the next section, and $B_r^+\subseteq \mb{R}^{n+1}_+$.

The structure of the paper is as follows: In Section \ref{prelim} we state the necessary extension theorems for the fractional Laplacian in order to define suitable weak solutions of \eqref{NS5} and some preliminary estimates involving extensions, in Section \ref{energysection} we derive an energy inequality for rescaled steady solutions of \eqref{NS5}, in Section \ref{decaysection} we prove decay estimates and an initial $\eps$-regularity result, Theorem \ref{main}, and in Section \ref{finalsection} we prove as a consequence our main $\eps$-regularity theorem, Theorem \ref{epsmain}, which implies the Hausdorff measure estimate of Theorem \ref{supermain} (details in the Appendix).

Regarding notation---we distinguish between derivatives in $\mb{R}^{n+1}_+$ and derivatives on the boundary $\mb{R}^n$ by placing a line over differential operators on $\mb{R}^{n+1}$. So for instance $\oDelta=\Delta+\partial_y^2=\sum_{i=1}^n\partial_{x_i}^2+\partial_y^2$. We also have $\oDelta_b U=\oDelta U+\frac{b}{y}\partial_y U=y^{-b}\text{Div }(y^b\onabla U)$ for functions $U$ on $\mb{R}^{n+1}_+$, following the standard notation for extension theorems as in \cite{CS,RY}. In many settings below it suffices to consider balls centered at the origin, in which case $B_r$ is understood to denote $B_r(0)$. Also we denote by $B_r^+(x)=B_r(x)\times[0,r)\subseteq\mb{R}^n\times\mb{R}_+$ and then have the similar convention $B_r^+=B_r\times[0,r)$.  When we refer to a cutoff function between open sets $\Omega_1$ and $\Omega$ for $\Omega_1\subseteq\Omega$ we mean a smooth nonnegative function which is identically $1$ on $\Omega_1$ and $0$ on $\Omega^c$. Further notation and conventions will be introduced below as needed. The dimension $n=5$ becomes important in Section \ref{finalsection} for compactness arguments needed to obtain the decay estimates, but until then we will mainly work without specifying the dimension $n$, though we will not track the dependence of constants on $n$. Finally, unless otherwise specified our discussion of \eqref{NS5} is restricted to $s\in(1,2)$.

To conclude this section we will give a discussion of our arguments and compare them with related arguments in previous papers. 

First for the non-steady equation \eqref{NS5time} when $s>1$, a major difficulty in studying partial regularity is that the natural generalization of the argument for $s< 1$ found in \cite{TY} fails. That argument depends on sequences of test functions constructed from the fractional heat kernel; however when $s>1$ the fractional heat kernel takes negative values and cannot be directly used to construct a sequence of test functions. The blow-up arguments of \cite{L} and \cite{LS} in the case $s=1$ avoid the direct use of the heat kernel, and \cite{CDM} also successfully developed a blow-up approach to extend the partial regularity theory for \eqref{NS5time} to $s>1$. The method of \cite{LS} is fairly general and has also been used recently to prove partial regularity for models closely related to the classical Navier--Stokes, again avoiding difficulties in constructing special test functions \cite{OR}.

In seeking to extend the partial regularity theory for the steady equation \eqref{NS5} to $s>1$ we encounter a similar problem; the test function arguments of \cite{TY2} and \cite{GM} for $s<1$ fail when $s>1$. Instead we will use extensively the methods developed by \cite{CDM} to study the non-steady equation \eqref{NS5time}. However, the difference is that their argument takes the external force $f$ to be zero, while we allow for suitable nonzero external forces $f$. In order to account for this we must appropriately adapt the blowup procedure of \cite{LS} to the fractional setting. We remark that the time-independence of \eqref{NS5} results in some of the arguments derived from \cite{CDM} in Section \ref{finalsection} used to obtain Theorem \ref{epsmain} as a consequence of Theorem \ref{main} becoming simpler with the help of estimates from \cite{TY2}. The following table illustrates the connection between Theorem \ref{supermain} and some previous results, in each case listed with the best known assumptions on $f$:
\begin{center}
  {\renewcommand{\arraystretch}{1.3}\begin{tabular}{ | l | l | l | l | l |}
    \hline
    $(-\Delta)^s$ power & \multicolumn{2}{l|}{\eqref{NS5time}, Non-steady $n=3$: $\mc{H}^{5-4s}(\mc{S})=0$} & \multicolumn{2}{l |}{\eqref{NS5}, Steady $n=5$: $\mc{H}^{7-6s}(\mc{S})=0$} \\ \hline
    $s<1$ & \cite{TY} & $f\in L^q_{t,x}$, $q>\frac{9+6s}{4s+1}$ & \cite{TY2}, \cite{GM} &  $f\in L^q_x$, $q>\frac{5}{2s}$\\ \hline
    $s=1$ & \cite{CKN},\cite{LS},\cite{Ku} & $f\in L^2_{t,x}$, $q\geq 2$ & \cite{Struwe} & $f\in L^q_x$, $q>\frac{5}{2}$\\ \hline
    $s>1$ & \cite{CDM} & $f=0$ & Theorem \ref{supermain} & $f\in L^q_x$, $q>\frac{5}{2s}$\\ \hline
  \end{tabular}}
\end{center}

Note that although \cite{Ku} allows for $f\in L^2_{t,x}$ for \eqref{NS5time} when $s=1$ and $n=3$, the earlier arguments in \cite{CKN} and \cite{LS} required $f\in L^q_{t,x}$ for $q>\frac{5}{2}$, which is exactly the range of $q$ for which $f\in L^q_x$ required by \cite{Struwe} for the steady equation \eqref{NS5} when $s=1$ and $n=5$.

\section{Suitable weak solutions and preliminary estimates}\label{prelim}

In this section we introduce the notion of suitable weak solutions of the steady equation \eqref{NS5} in Definition \ref{suitable}. In order to do this we must first describe the extension theorems for the fractional Laplacian of Caffarelli--Silvestre \cite{CS} and R. Yang \cite{RY}.  To conclude we quote some general integral estimates of these extensions from \cite{CDM} adapted to the $n$-dimensional setting which will be useful later.

\subsection{Extension theorems for $(-\Delta)^s$}

\begin{theorem}[Caffarelli--Silvestre \cite{CS}]\label{CaffS}
    Let $w\in \dot{H}^{\tilde{s}}(\mb{R}^n)$ with $\tilde{s}\in(0,1)$ and set $\tilde{a}=1-2\tilde{s}$. Then there is a unique extension $w^\flat\in \dot{H}^1(\mb{R}^{n+1}_+,y^{\tilde{a}})$ which satisfies
    \begin{align}
        \oDelta_{\tilde{a}} w^\flat(x,y)=0\quad\text{and the boundary condition}\quad w^\flat(x,0)=w(x),\notag
    \end{align}
This extension can be characterized by
\begin{align}
	w^\flat(x,y)=\int_{\mb{R}^n}Q(x-\xi,y)w(\xi)\ d\xi, \quad\text{or}\quad\widehat{w^\flat}(\xi,y)=\hat{w}(\xi)\psi(|\xi|y),\notag
\end{align}
where $Q(x,y)=\ol{C}_{n,\tilde{s}}\frac{y^{2\tilde{s}}}{(|x|^2+y^2)^{\frac{n+2\tilde{s}}{2}}}$ and $\psi\in W^{1,2}(\mb{R}_+,y^{\tilde{a}})$ minimizes a variational problem (see \cite{CS}). Moreover there exists a constant $C_{n,\tilde{s}}$ depending on $n$ and $\tilde{s}$ such that:
    \begin{enumerate}[(a)]
        \item The fractional Laplacian $(-\Delta)^{\tilde{s}}$ is given by $(-\Delta)^{\tilde{s}}w(x)=-C_{n,\tilde{s}}\lim_{y\rightarrow 0}y^{\tilde{a}}\partial_y w^\flat(x,y)$.
        \item The following energy identity holds:
        \begin{align} 
            \int_{\mb{R}^n}|(-\Delta)^{\frac{\tilde{s}}{2}} w|^2\ dx=\int_{\mb{R}^n}|\xi|^{2\tilde{s}-2}|\hat{w}(\xi)|^2\ d\xi=C_{n,\tilde{s}}\int_{\mb{R}^{n+1}_+}y^{\tilde{a}}|\onabla w^\flat|^2\ dx\ dy.\notag
        \end{align}
        \item The following inequality holds for every extension $\tilde{v}\in H^1(\mb{R}^{n+1}_+,y^{\tilde{a}})$ of $w$:
        \begin{align}
            \int_{\mb{R}^{n+1}_+}y^{\tilde{a}}|\onabla w^\flat|^2\ dx\ dy\leq\int_{\mb{R}^{n+1}_+}y^{\tilde{a}}|\onabla \tilde{v}|^2\ dx\ dy.\notag
        \end{align}
    \end{enumerate}
\end{theorem}

\begin{theorem}[R. Yang \cite{RY}, see also \cite{CG,CRY,CDM}]\label{Rays}
    Let $u\in \dot{H}^s(\mb{R}^n)$ with $s\in(1,2)$ and set $b=3-2s$. Then there is a unique extension $u^*$ of $u$ in $L^2_{\text{loc}}(\mb{R}^{n+1}_+,y^b)$ with $\oDelta_b u^*\in L^2(\mb{R}^{n+1}_+,y^b)$ which satisfies
    \begin{align}
        \oDelta_b^2 u^*(x,y)=0,\quad\text{and the boundary conditions}\quad
	\begin{cases}
        u^*(x,0)=u(x),\notag
        \\
        \lim_{y\rightarrow 0} y^{1-s}\partial_y u^*(x,y)=0.\notag
	\end{cases}
    \end{align}
	This extension can be characterized by
	\begin{align}
	u^*(x,y)=\int_{\mb{R}^n}P(x-\xi,y)u(\xi)\ d\xi, \quad\text{or}\quad\widehat{u^*}(\xi,y)=\hat{u}(\xi)\phi(|\xi|y),\notag
	\end{align}
	where $P(x,y)=\ol{c}_{n,s}\frac{y^{2s}}{(|x|^2+y^2)^{\frac{n+2s}{2}}}$ and $\phi\in W^{2,2}(\mb{R}_+,y^b)$ minimizes a variational problem (see \cite{RY}).
    Moreover there exists a constant $c_{n,s}$ depending on $n$ and $s$ with the following properties:
    \begin{enumerate}[(a)]
        \item The fractional Laplacian $(-\Delta)^s u$ is given by $(-\Delta)^s u(x)=c_{n,s}\lim_{y\rightarrow 0} y^b\partial_y\oDelta_b u^*(x,y)$.
        \item The following energy identity holds:
        \begin{align}
            \int_{\mb{R}^n}|(-\Delta)^{\frac{s}{2}}u|^2\ dx=\int_{\mb{R}^n}|\xi|^{2s}|\hat{u}(\xi)|^2\ d\xi=c_{n,s}\int_{\mb{R}^{n+1}_+}y^b|\oDelta_b u^*|^2\ dx\ dy.\notag
        \end{align}
        \item The following inequality holds for every extension $v\in L^2_{\text{loc}}(\mb{R}^{n+1}_+,y^b)$ of $U$ with $\oDelta_b v\in L^2(\mb{R}^{n+1}_+,y^b)$:
        \begin{align}
            \int_{\mb{R}^{n+1}_+}y^b|\oDelta_b u^*|^2\ dx\ dy\leq\int_{\mb{R}^{n+1}_+}y^b|\oDelta_b v|^2\ dx\ dy.\notag
        \end{align}
    \end{enumerate}
\end{theorem}

\begin{Rem}
	We will apply Theorem \ref{Rays} in Sections \ref{prelim}--\ref{decaysection} to $u\in\dot{H}^s$ in order to prove Theorem \ref{main}, and apply Theorem \ref{CaffS} to $\nabla u\in \dot{H}^{s-1}$ in order to prove Theorem \ref{epsmain} as a consequence of Theorem \ref{main} in Section \ref{finalsection}. So for our applications $\tilde{s}=s-1$ so that $\tilde{a}=3-2s=b$, and therefore throughout this paper we have only the weight $y^b$ on $\mb{R}^{n+1}_+$ regardless of the extension theorem used.
\end{Rem}

\subsection{Suitable weak solutions}

Now we will define the suitable weak solutions of \eqref{NS5}.  Before stating the definition we introduce the function space to which we assume the external force $f$ belongs. For $\gamma>0$, we define
\begin{align}
    M_{2s,\gamma}(\mb{R}^n)=\{f\in L^2_{\text{loc}}(\mb{R}^n,\mb{R}^n):\ c_\gamma(f)<\infty\},\notag
\end{align}
where $c_\gamma(f)=\sup\left\{\frac{1}{R^{\gamma-2s}}\left(\textmean{B_R(x)}|f|^2\ dx\right)^{\frac{1}{2}}:\ B_R(x)\subseteq\mb{R}^n\right\}$.

Observe that if $f\in L^q(\mb{R}^5)$ for a $q\geq 2$, then $\int_{B_R(x)}|f|^q\ dx<\infty$ for all $B_R(x)\subseteq\mb{R}^n$, so that
\begin{align}
    R^{-\gamma+2s}\left(\mean{B_R(x)}|f|^2\ dx\right)^{1/2}&\leq R^{-\gamma+2s}\left(\mean{B_R(x)} |f|^q\ dx\right)^{1/q}=R^{-\gamma+2s-n/q}\left(\int_{B_R(x)} |f|^q\ dx\right)^{1/q},\notag
\end{align}
which is bounded if $-\gamma+2s-\frac{n}{q}=0$. Thus if the force $f$ belongs to $L^q(\mb{R}^n)$ for some $q>\frac{n}{2s}$ then $f\in M_{2s,\gamma}$ for some $\gamma>0$. We will prove our $\eps$-regularity results for $f\in M_{2s,\gamma}$ which will imply the weaker hypotheses on $f$ in Theorems \ref{supermain} and \ref{epsmain}.

\begin{Def}[Suitable weak solutions of \eqref{NS5}]\label{suitable}
    We call a pair $(u,p)$ with $u\in \dot{H}^s(\mb{R}^5)$ and $p\in L^{\frac{5}{5-2s}}(\mb{R}^5)$ a suitable weak solution of \eqref{NS5} with force $f$ if it solves \eqref{NS5} in the weak sense and also satisfies the following energy inequality:
    \begin{align}
        c_{n,s}\int_{\mb{R}^{n+1}_+}y^b|\ol{\Delta}_b u^*|^2\Phi\ dx\ dy\leq&\int_{\mb{R}^n}\left(\frac{|u|^2}{2}+p\right)u\cdot\nabla\varphi\ dx\label{ssen}
        \\
        &\quad-c_{n,s}\int_{\mb{R}^{n+1}_+}y^b\ol{\Delta}_b u_i^*(2\ol{\nabla} u_i^*\cdot\ol{\nabla}\Phi+u_i^*\ol{\Delta}_b\Phi)\ dx\ dy+\int_{\mb{R}^n} f\cdot u\varphi\ dx.\notag
    \end{align}
    Here $\Phi\in C_0^\infty(\mb{R}^{n+1}_+)$ with $\partial_y\Phi(\cdot,0)=0$ in $\mb{R}^n$, and $\Phi(x,0)=\varphi(x)$.
\end{Def}

If $(u,p)$ is sufficiently smooth then \eqref{ssen} follows from integration by parts using the extension properties of Theorem \ref{Rays}. The existence of suitable weak solutions of \eqref{NS5} when $s<1$ and $n=3$ was shown in \cite[Proposition 5.3]{TY2} by generalizing the proof using the Galerkin method for $s=1$ of \cite[Section 2.1]{Temam}. The same method extends to our setting when $s>1$ (in fact the case $s<1$ presents more difficulties for the construction) and we quote it below.

\begin{theorem}[Existence of suitable weak solutions {{\cite[Proposition 5.3]{TY2}}}]
	Given an external force $f\in L^{\frac{10}{5+2s}}(\mb{R}^5)$, there exists a suitable weak solution $(u,p)$ of \eqref{NS5}.
\end{theorem}

\subsection{Extension estimates}

We now quote below several inequalities from Colombo--De Lellis--Massaccesi \cite{CDM} comparing a function $u$ and its extension $u^*$ as defined in Theorem \ref{Rays} as well as its derivatives, which will be useful in our subsequent arguments. In \cite{CDM} they are proved for space dimension $3$ but they extend easily to arbitrary space dimension $n$ as stated below. The first lemma below provides control of the extension function $u^*$ by the boundary function $u$.

\begin{Lem}[{{\cite[Lemma 3.4]{CDM}}}]\label{3.4}
    Let $s\in(1,2)$ and let $P$ be the kernel $P(x,y)=\bar{c}_{n,s}\frac{y^{2s}}{(|x|^2+y^2)^{\frac{n+2s}{2}}}$ of Theorem \ref{Rays}. There exists a $C>0$ such that if $u\in L^2(B_1)\cap L^1_{\text{loc}}(\mb{R}^n)$ and
    \begin{align}
        \sup_{R\geq 1} R^{-3s}\left(\mean{B_R}|u|\ dx\right)^2<\infty,\notag
    \end{align}
    then the associated extension $u^*(x,y)=(P(\cdot,y)\ast u)(x)$ belongs to $L^2_{\text{loc}}(\mb{R}^{n+1}_+,y^b)$ and satisfies
    \begin{align}
        \int_{B_1^+}y^b|u^*|^2\ dx\ dy\leq C\left(\int_{B_2}|u|^2\ dx+\sup_{R\geq 1} R^{-3s}\left(\mean{B_R}|u|\ dx\right)^2\right).
    \end{align}
\end{Lem}

Next we have an interpolation lemma estimating the derivatives of $u^*$.

\begin{Lem}[{{\cite[Lemma 3.3]{CDM}}}]\label{3.3}
    Given $s\in(1,2)$, there exists a $C>0$ such that for $\Phi\in C_0^\infty(\mb{R}^{n+1})$, $u\in \dot{H}^s(\mb{R}^n)$, and $u^*$ its extension as in Theorem \ref{Rays} along with $\eps\in(0,1)$ and $r\in(0,\infty)$ the following inequalities hold:
\begin{subequations}
    \begin{align}
        \int_{\mb{R}^{n+1}_+}y^b|\onabla u^*|^2\varphi^2\ dx\ dy\leq \eps\int_{\mb{R}^{n+1}_+} y^b|\oDelta_b u^*|^2\Phi^2\ dx\ dy+\frac{C}{\eps}\int_{\mb{R}^{n+1}_+}y^b|u^*|^2(\Phi^2+|\onabla\Phi|^2)\ dx\ dy,
        \\
        \int_{B_r^+}y^b|\onabla u^*|^2\ dx\ dy\leq C\left(\int_{B_{2r}^+}y^b|\oDelta_b u^*|^2\ dx\ dy\right)^{\frac{1}{2}}\left(\int_{B_{2r}^+}y^b|u^*|^2\ dx\ dy\right)^{\frac{1}{2}}+\frac{C}{r^2}\int_{B_{2r}^+}y^b|u^*|^2\ dx\ dy.
    \end{align}
\end{subequations}
\end{Lem}

Finally we have a Sobolev-type inequality resulting from the application of Lemmas \ref{3.4} and \ref{3.3} to the expression in $u^*$ of the $\dot{H}^s$ norm of $u$ from Theorem \ref{Rays}. We will also restate some of the proof from \cite{CDM} but for arbitrary dimension $n$ because we will need to refer to the proof later.

\begin{Lem}[Lemma 4.4 of \cite{CDM}]\label{4.4}
    Given $s\in(1,2)$ and $r\in(0,1)$, there exists a $C>0$ such that for any $u\in \dot{H}^s(\mb{R}^n)$,
    \begin{align}
        \|u\|^2_{L^{\frac{10}{5-2s}}(B_r)}\leq C\left(\int_{B_1^+}y^b|\oDelta_b u^*|^2\ dx\ dy+\sup_{R\geq\frac{1}{4}}R^{-3s}\mean{B_R}|u|^2\ dx\right).
    \end{align}
\end{Lem}
\begin{proof}
    Fix $r\in(0,1)$ and let $\Phi$ be a smooth cutoff function between $B_r^+$ and $B_1^+$. We denote $\Phi(x,0)=\varphi(x)$, and assume that the cutoff function $\Phi|_{y<\frac{1}{2}}$ is independent from the variable $y$, so that $\onabla\Phi=\nabla\Phi$ in the set $\{y<\frac{1}{2}\}$. We can estimate
    \begin{align}
        \int_{\mb{R}^n}|(-\Delta)^{s/2}(u\varphi)|^2\ dx&=C\int_{\mb{R}^{n+1}_+}y^b|\oDelta_b(u\varphi)^*|^2\ dx\ dy\leq C\int_{\mb{R}^{n+1}_+} y^b|\oDelta_b(u^*\Phi)|^2\ dx\ dy\notag\notag
        \\
        &\leq C\int_{B_1^+} y^b(|\oDelta_b u^*|^2\Phi^2+|\onabla u^*|^2|\onabla\Phi|^2+|u^*|^2|\oDelta_b\Phi|^2)\ dx\ dy.\notag
    \end{align}
    By Lemma \ref{3.3} we can estimate the middle term containing $\onabla u^*$, taking $\eps=\frac{1}{2}$ so that
    \begin{align}
        \int_{\mb{R}^{n+1}_+}y^b|\oDelta_b(u^*\Phi)|^2\ dx\ dy&\leq C\int_{B_1^+} y^b(|\oDelta_b u^*|^2+|u^*|^2)\ dx\ dy\notag
        \\
        &\leq C\left(\int_{B_1^+} y^b|\oDelta_b u^*|^2\ dx\ dy+\sup_{R\geq\frac{1}{4}}R^{-3s}\mean{B_R}|u|^2\ dx\right),\notag
    \end{align}
    where in the last line we have applied Lemma \ref{3.4}. The conclusion follows from the Sobolev embedding $\dot{H}^s(\mb{R}^n)\hookrightarrow L^{\frac{2n}{n-2s}}(\mb{R}^n)$.
\end{proof}

\section{An energy inequality for rescaled solutions}\label{energysection}

In this section we consider rescalings of a suitable weak solution $(u,p)$ of \eqref{NS5} with force $f$, which will be important in Section \ref{decaysection}. Our choice of rescaling though related is distinct from that in \cite{CDM} and inspired by the methods in \cite{LS}. Let $R>0$, $M\in\mb{R}^n$, $L\in\mb{R}_+$, and $\tilde{p}\in\mb{R}$, and define
    \begin{align}
        v(x)=\frac{u(Rx)-M}{L},\quad q(x)=\frac{p(Rx)-\tilde{p}}{L} R^{2s-1},\quad g(x)=f(Rx).\label{rdef}
    \end{align}
It follows that in the weak sense $(v,q)$ and $g$ solve
\begin{align}
    \begin{cases}
        (-\Delta)^s v+R^{2s-1} M\nabla v+ R^{2s-1}L v\nabla v+\nabla q=\frac{R^{2s-1}}{L} g,
        \\
        \text{div }u=0.\label{rNS5}
    \end{cases}
\end{align}
We will check that such solutions satisfy an energy inequality by a careful change of variables, similar to the procedure detailed in \cite{CDM}.

\begin{Lem}\label{3.1}
    Let $(u,p)$ be a suitable weak solution of \eqref{NS5} with force $f$ and define $(v,q)$ and $g$ as in \eqref{rdef}.
    Then $(v,q)$ with force $g$ satisfies \eqref{rNS5} weakly and further satisfies the following energy inequality for any $\Phi\in C_0^\infty(B_1^+)$ with $\Phi(x,0)=\varphi(x)$:
    \begin{align}
        c_{n,s}\int_{B_1^+} y^b|\oDelta_b v^*|^2\Phi\ dx\ dy&\leq R^{2s-1}L\int_{B_1}\frac{|v|^2}{2}v\cdot\nabla\varphi\ dx+\int_{B_1}qv\cdot\nabla\varphi\ dx\notag
	\\
	&\quad+R^{2s-1}M\cdot\int_{B_1}\frac{|v|^2}{2}\nabla\varphi\ dx+\frac{R^{2s}}{L}\int_{B_1} g\cdot v\varphi\ dx\notag
        \\
        &\quad-c_{n,s}\int_{B_1^+} y^b\oDelta_b v^*(2\onabla v^*\onabla\Phi+ v^*\oDelta_b\Phi)\ dx\ dy.\label{rNS5en}
    \end{align}
\end{Lem}
\begin{proof}
	First, we may assume $\tilde{p}=0$, since $\tilde{p}$ would only add a term to the inequality which vanishes,
    \begin{align}
        \int_{\mb{R}^n}\tilde{p} u(x)\cdot\nabla\varphi(x)\ dx=0.\notag
    \end{align}
    From our test function $\Phi\in C_0^\infty(B_1^+)$ we define $\Psi(Rx,Ry)=\Phi(x,y)$ so that $\Psi(x,y)\in C_0^\infty(B_R^+)$. We let $\psi(x)=\Psi(x,0)$ and $\varphi(x)=\Phi(x,0)$. We now substitute $v$ and $v^*$ where possible for $u$ and $u^*$ in the energy inequality \eqref{ssen}. First, using that $\oDelta_b v^*(x,y)=\frac{R^2}{L}(\oDelta_b u^*)(Rx,Ry)$ we have
    \begin{align}
        c_{n,s}\int_{\mb{R}^{n+1}_+} y^b|\oDelta_b u^*|^2\Psi\ dx\ dy= c_{n,s}R^{n-2s} L^2\int_{B_1^+} y^b|\oDelta_b v^*|^2\Phi\ dx\ dy.\label{LHSineq1}
    \end{align}
    Next we observe that
    \begin{align}
        &\int_{\mb{R}^n}\left(\frac{|u|^2}{2}+p\right) u\cdot\nabla\psi\ dx\notag
        \\
        &=R^{n-1}\left[\int_{B_1} L^2\left(\frac{L|v|^2}{2}+ R^{1-2s} q\right)v\cdot\nabla\varphi\ dx\right.+\int_{B_1} L^2\frac{|v|^2}{2} M\cdot\nabla\varphi\ dx\label{upsubidentity}
        \\
        &\quad+\int_{B_1} (M\cdot Lv) (Lv+M)\cdot\nabla\varphi\ dx+\int_{B_1}\frac{|M|^2}{2}(Lv+M)\cdot\nabla\varphi\ dx+\left.\int_{B_1} LR^{1-2s}q M\cdot\nabla\varphi\ dx\right].\notag
    \end{align}
    The second to last term on the right vanishes because $v$ is divergence free. To address the two other terms on the last line we use the equation satisfied by $u$:
    \begin{align}
        0&=\int_{\mb{R}^n} u\nabla\psi\cdot u+p\nabla\psi+ f\psi-\psi(-\Delta^s) u\ dx\notag
        \\
        &=R^{n-1}\int_{B_1} (Lv+M)\nabla\varphi\cdot(Lv+M)\ dx+R^{n-1}\int_{B_1}\nabla\varphi LR^{1-2s} q\ dx+R^n\int_{B_1}g\varphi\ dx\notag
        \\
        &\quad-c_{n,s}R^{n+1} R^b R^{-4}L\int_{B_1^+}y^b\oDelta_b v^*\oDelta_b\Phi\ dx\ dy,\label{weakuidentity}
    \end{align}
	where we have used the properties of $u^*$ and $\Phi$ to observe that
	\begin{align}
	\int \psi(-\Delta)^s u\ dx=c_{n,s}\int_{\mb{R}^n}\lim_{y\rightarrow 0} y^b\psi\partial_y\oDelta_b u^*\ dx&=-c_{n,s}\int_{\mb{R}^{n+1}_+} y^b\onabla\oDelta_b u^*\cdot\onabla\Psi\ dx\ dy\notag
	\\
	&=c_{n,s}\int_{\mb{R}^{n+1}_+} y^b\oDelta_b u^*\oDelta_b\Psi\ dx\ dy.\notag
	\end{align}
    We take the scalar product of $M$ with \eqref{weakuidentity} and combine this with \eqref{upsubidentity} to see that
    \begin{align}
        &\int_{\mb{R}^n}\left(\frac{|u|^2}{2}+p\right) u\cdot\nabla\psi\ dx\notag
        \\
        &=R^{n-1}\left[\int_{B_1} L^2\left(\frac{L|v|^2}{2}+ R^{1-2s} q\right)v\cdot\nabla\varphi\ dx\right.\left.+\int_{B_1} L^2\frac{|v|^2}{2} M\cdot\nabla\varphi\ dx\right]\notag
        \\
        &\quad-R^n\int_{B_1}M\cdot g\varphi\ dx+ c_{n,s} R^{n-2s}\int_{B_1^+} y^b M\oDelta_b v^*\oDelta_b\Phi\ dx\ dy.\label{RHSvp1}
    \end{align}
    Finally we have
    \begin{align}
        c_{n,s}\int_{\mb{R}^{n+1}_+} y^b\oDelta_b u^*(2\onabla\Psi\onabla u^*+u^*\oDelta_b\Psi)\ dx\ dy&=c_{n,s} R^{n-2s} L^2\int_{B_1^+} y^b\oDelta_b v^*(2\onabla\Phi\onabla v^*)\ dx\ dy\label{extsubdirect}
        \\
        &\quad+ c_{n,s} R^{n-2s}L \int_{B_1^+} y^b\oDelta_b v^*(Lv^*+M)\oDelta_b\Phi\ dx\ dy.\notag
    \end{align}
    Therefore, combining \eqref{ssen} which is satisfied by $(u,p)$ with \eqref{LHSineq1}, \eqref{RHSvp1}, and \eqref{extsubdirect} gives the desired energy inequality \eqref{rNS5en}.
\end{proof}

We call these $(v,q)$ suitable weak solutions of \eqref{rNS5} with force $g$. Now we simplify \eqref{rNS5en} by applying the extension estimates of Section \ref{prelim} to establish control of a local energy term $\oDelta_b v^*$ by only $v^*$, $v$, $q$, and $g$, without any derivatives of $v^*$.

\begin{Lem}\label{3.2}
    Let $M\in\mb{R}^n$, and let $(v,q)$ with force $g$ be a suitable weak solution of \eqref{rNS5}. Then we have
    \begin{align}
        \int_{B_{3/4}^+}y^b|\oDelta_b v^*|^2\ dx\ dy&\leq C R^{2s-1}\int_{B_1}L|v|^3\ dx+C\int_{B_1}|q||v|\ dx+CR^{2s-1}|M|\int_{B_1}|v|^2\ dx\notag
        \\
        &\quad+C\int_{B_1^+}y^b|v^*|^2\ dx\ dy+C\frac{R^{2s}}{L}\int_{B_1}|g||v|\ dx.\label{rescaledenineq}
    \end{align}
\end{Lem}
\begin{proof}
    We take $r,s\in(3/4,1)$ with $r<s$ and $r'=(2r+s)/3$, $s'=(r+2s)/3$ and a cutoff function $\Phi$ supported in $B^+_{s'}$ which is $1$ on $B_{r'}^+$ satisfying
    \begin{align}
        |\partial_y\Phi|+|\onabla\Phi|\leq\frac{C}{s-r},\quad|\onabla^2\Phi|\leq\frac{C}{(s-r)^2}.\notag
    \end{align}
    We further assume that $\partial_y\Phi=0$ on $\{y<\frac{1}{2}\}$, so that $|\oDelta_b\Phi|\leq C(r-s)^{-2}$ as well, and take this $\Phi$ to be the test function in \eqref{rNS5en}. Setting $h(r)=c_{n,s}\int_{B_r^+}y^b|\oDelta_b v^*|^2\ dx\ dy$, we then have
    \begin{align}
        h(r)&\leq C R^{2s-1}\int_{B_{s'}\backslash B_{r'}}\frac{L|v|^3}{s-r}\ dx+C\int_{B_{s'}\backslash B_{r'}}\frac{|q||v|}{s-r}\ dx+CR^{2s-1}|M|\int_{B_{s'}\backslash B_{r'}}\frac{|v|^2}{s-r}\ dx\notag
        \\
        &\quad+C\int_{B_{s'}^+\backslash B_{r'}^+}\frac{y^b|\oDelta_b v^*||\onabla v^*|}{s-r}+\frac{y^b|\oDelta_b v^*||v^*|}{(s-r)^2}\ dx\ dy+C\frac{R^{2s}}{L}\int_{B_{s'}}|g||v|\ dx.\label{en2start}
    \end{align}
    We now look at the two integrals in $\mb{R}^{n+1}_+$. By Young's inequality we have
    \begin{align}
        \int_{B_{s'}^+\backslash B_{r'}^+}\frac{y^b|\onabla v^*||\oDelta_b v^*|}{s-r}\ dx\ dy&\leq \int_{B_{s'}^+\backslash B_{r'}^+} y^b|\oDelta_b v^*|^2\ dx\ dy+\frac{C}{(s-r)^2}\int_{B_{s'}^+\backslash B_{r'}^+}y^b|\onabla v^*|^2\ dx\ dy.\notag
        \\
        &\leq (h(s)-h(r))+\frac{C}{(s-r)^2}\int_{B_{s'}^+\backslash B_{r'}^+}y^b|\onabla v^*|^2\ dx\ dy,\label{en2Young}
    \end{align}
    and similarly
    \begin{align}
        \int_{B_{s'}^+\backslash B_{r'}^+}\frac{y^b|\oDelta_b v^*||v^*|}{(s-r)^2}\ dx\ dy&\leq\int_{B_{s'}^+\backslash B_{r'}^+}y^b|\oDelta_b v^*|^2\ dx\ dy+\frac{1}{(s-r)^4}\int_{B_{s'}^+\backslash B_{r'}^+}y^b|v^*|^2\ dx\ dy\notag
        \\
        &\leq (h(s)-h(r))+\frac{1}{(s-r)^4}\int_{B_1^+}y^b|v^*|^2\ dx\ dy.
    \end{align}
    To control the integral involving $\onabla v^*$ in \eqref{en2Young} we apply Lemma \ref{3.3} with a test function $\Psi\in C_0^\infty(\mb{R}^{n+1})$ satisfying $\text{supp }\Psi\cap\mb{R}^{n+1}_+\subset B_s^+\backslash B_r^+$, $\Psi\equiv 1$ in $B_{s'}^+\backslash B_{r'}^+$, and $|\onabla\Psi|\leq\frac{C}{s-r}$ to find that
    \begin{align}
        \frac{1}{(s-r)^2}\int_{B^+_{s'}\backslash B^+_{r'}} y^b|\onabla v^*|^2\ dx\ dy&\leq\int_{B_s^+\backslash B_r^+}y^b|\oDelta v^*|^2\ dx\ dy+\frac{C}{(s-r)^4}\int_{B_1^+}y^b|v^*|^2\ dx\ dy\notag
        \\
        &\leq (h(s)-h(r))+\frac{C}{(s-r)^4}\int_{B_1^+} y^b|v^*|^2\ dx\ dy.\label{en2end}
    \end{align}
    Combining \eqref{en2start}--\eqref{en2end} therefore gives
    \begin{align}
        h(r)&\leq C(h(s)-h(r))+CR^{2s-1}\frac{1 }{s-r}\int_{B_{s'}\backslash B_{r'}}L|v|^3\ dx+C\frac{1}{s-r}\int_{B_{s'}\backslash B_{r'}}|q||v|\ dx\notag
        \\
        &\quad+CR^{2s-1}|M|\frac{1}{s-r}\int_{B_{s'}\backslash B_{r'}}|v|^2\ dx+C\frac{R^{2s}}{L}\int_{B_1}|g||v|\ dx\notag
        \\
        &\quad+C\frac{1}{(s-r)^4}\int_{B_1^+}y^b|v^*|^2\ dx\ dy,\notag
    \end{align}
    so that applying Lemma 6.1 of \cite{Giusti} as in the proof of Lemma 3.2 in \cite{CDM} we obtain the conclusion of the lemma.
\end{proof}

\section{Decay estimates}\label{decaysection}

We now prove decay estimates for suitable weak solutions of \eqref{NS5} via a rescaling argument using excess and energy quantities as well as some associated estimates from Colombo--De Lellis--Massaccesi \cite{CDM}. Despite some similarities, however, there are important differences between their approach and ours because we must deal with a nonzero external force $f$ while \cite{CDM} works with $f=0$. In order to accommodate nonzero $f$ we will adapt to the fractional setting the blow-up arguments of Ladyzhenskaya--Seregin \cite{LS}. First we will define the energy $\ol{E}$ used in this section in order to state our main result. In what follows we fix the exponent $s\in(1,2)$ in \eqref{NS5} as well as the $\gamma>0$ for which $f\in M_{2s,\gamma}$, and do not track below the dependence on $s$ or $\gamma$ of the constants below.

\begin{Def}[Energy]
	Let $\ol{E}(u,p;x,r)=\ol{E^V}(u;x,r)+\ol{E^P}(p;x,r)+\ol{E^{nl}}(u;x,r)$, where
    \begin{align}
        \ol{E^V}(u;x,r)&=\left(\mean{B_r(x)}|u|^3\ dx\right)^{\frac{1}{3}}\notag
        \\
        \ol{E^P}(p;x,r)&=r^{2s-1}\left(\mean{B_r(x)}|p|^{\frac{3}{2}}\ dx\right)^{\frac{2}{3}}\notag
        \\
        \ol{E^{nl}}(u;x,r)&=\left(\sup_{R\geq\frac{1}{4}r}\left(\frac{r}{R}\right)^{3s}\mean{B_R(x)}|u|^2\ dx\right)^{\frac{1}{2}}.\notag
    \end{align}
\end{Def}
Here, $E^V$ denotes an energy of the velocity field $u$, $E^P$ denotes an energy of the pressure $p$, and $E^{nl}$ denotes a ``nonlocal'' energy involving the velocity field, in the sense that it depends on $u$ on all of $\mb{R}^5$. Above $(h)_{B_r(x)}$ denotes $\textmean{B_r(x)}h\ dx$. Many of our statements do not depend on the point $x$, so often we take $x=0$, denoting $\ol{E}(u,p;r)=\ol{E}(u,p;0,r)$ and so on. For averages we will also write $(h)_r$ to denote $(h)_{B_r}$, an average across a ball of radius $r$ centered at the origin. The main result of this section is Theorem \ref{main} below.

\begin{Thm}\label{main}
    There exists $\ol{\eps_0}$ such that if $(u,p)$ is a suitable weak solution of \eqref{NS5} with force $f$ and 
    \begin{align}
        \liminf_{r\rightarrow 0}r^{2s-1}\ol{E}(u,p;x,r)<\ol{\eps_0},
    \end{align}
    then $x$ is a regular point of $u$.
\end{Thm}

The remainder of this section is dedicated to proving Theorem \ref{main} as follows. In the first part of this section we prove the main decay estimate Proposition \ref{6.1} which has a smallness hypothesis on $E$ with the help of a regularity lemma, Lemma \ref{5.1}, which we first assume in the proof of Proposition \ref{6.1} and then justify afterwards. In the second part of this section we prove Theorem \ref{main} from Proposition \ref{6.1}. First we iterate Proposition \ref{6.1} to prove Lemma \ref{2.5}. Then Lemma \ref{2.6} shows that the original smallness hypothesis on $E$ can also hold with $E$ replaced with $\ol{E}$, and finally Theorem \ref{main} follows from this by using the rescaling properties of \eqref{NS5} and the energy $\ol{E}$.

\subsection{Proof of Proposition \ref{6.1}}

Before we state Proposition \ref{6.1} we need to define the excess quantities $E$.

\begin{Def}[Excess]
	Let $E(u,p;x,r)=E^V(u;x,r)+E^P(p;x,r)+E^{nl}(u;x,r)$, where
	 \begin{align}
        E^V(u;x,r)&=\left(\mean{B_r(x)}|u-(u)_{B_r(x)}|^3\ dx\right)^{\frac{1}{3}}\notag
        \\
        E^P(p;x,r)&=r^{2s-1}\left(\mean{B_r(x)}|p-(p)_{B_r(x)}|^{\frac{3}{2}}\ dx\right)^{\frac{2}{3}}\notag
        \\
        E^{nl}(u;x,r)&=\left(\sup_{R\geq\frac{1}{4}r}\left(\frac{r}{R}\right)^{3s}\mean{B_R(x)}|u-(u)_{B_r(x)}|^2\ dx\right)^{\frac{1}{2}}.\notag
	\end{align}
\end{Def}

\begin{Prop}[Main decay estimate]\label{6.1}
    Given $\theta,M,\beta$ satisfying
    \begin{align}
        0<\theta\leq\frac{1}{2},\quad M\geq 3,\quad 0<\beta<\gamma,\notag
    \end{align}
    there exist $\eps_1>0$ and $R_1>0$ depending on $\theta,M,\beta$ so that if $(u,p)$ is a suitable weak solution of \eqref{NS5} with force $f\in M_{2s,\gamma}$ then
    \begin{align}
        \begin{cases}
            A(u;r)=r^{2s-1}|(u)_r|<M
            \\
            E(u,p;r)+c_\gamma(f) r^\beta<\eps_1,\notag
        \end{cases}
    \end{align}
    for any $r<R_1$ implies the decay estimate
    \begin{align}
        E(u,p;\theta r)\leq c_1\theta^{\alpha_1}(E(u,p;r)+c_\gamma(f)r^\beta).
    \end{align}
    Here $\alpha_1$ can be fixed to be any number in $(0,1)$ and $c_1$ depends only on $M$ and the choice of $\alpha_1$.
\end{Prop}

We will prove the proposition by contradiction. First we will set up the argument. Let $\theta,M,\beta$ be chosen as specified, and suppose that the proposition is false. Then there is a sequence $\{(u_k,p_k)\}$ of steady suitable weak solutions to \eqref{NS5} with force $f_k$ and $r_k\rightarrow 0$ satisfying
\begin{subequations}
    \begin{align}
           E(u_k,p_k;r_k)+d_kr_k^\beta=\eps_k\rightarrow 0,\label{contr1}
           \\
           E(u_k,p_k;\theta r_k)\geq c_1\theta^{\alpha_1}\eps_k,\label{contra}
    \end{align}
\end{subequations}
    where $d_k=c_\gamma(f_k)$, and $c_1$ will be chosen to give a contradiction. By \eqref{contra}, $(u_k,q_k)$ does not have excess decay, and we will obtain a contradiction to this in three steps: First in Lemma \ref{compact} we rescale $(u_k,p_k)$ and $f_k$ to obtain a sequence $(v_k,q_k)$ which also does not have excess decay but converges to some $(v,q)$, second in Lemma \ref{5.1} we study the linear equation satisfied by $(v,q)$ and conclude that it does have excess decay, and third we complete the proof of Proposition \ref{6.1} by using the convergence $(v_k,q_k)\rightarrow(v,q)$ to see that $(v_k,q_k)$ does in fact have excess decay for $k$ sufficently large, which contradicts \eqref{contra}.

\begin{Lem}\label{compact}
	Let $\{(u_k,p_k)\}$ be a sequence of steady suitable weak solutions to \eqref{NS5} with force $f_k$ and $r_k\rightarrow 0$ satisfying \eqref{contr1}. Define
\begin{align}
        v_k(x)=\frac{u_k(r_kx)-(u_k)_{r_k}}{\eps_k},\quad q_k(x)=\frac{p_k(r_kx)-(p_k)_{r_k}}{\eps_k}r_k^{2s-1},\quad        g_k(x)=f_k(r_kx).\label{rescale}
\end{align}
Then there exists a subsequence (which we also denote $\{(v_k,q_k)\}$) which converges in the sense of distributions to a pair $(v,p)$ with $v\in L^2_{\text{loc}}(\mb{R}^5)$ and also
\begin{align}
	v_k&\rightarrow v\quad\text{in}\quad L^3(B_{\frac{1}{2}},\mb{R}^5),\quad q_k\rightarrow q\quad\text{in}\quad L^{\frac{3}{2}}(B_{\frac{1}{2}}).\notag
\end{align}
Furthermore $(v,q)$ satisfies for all $\varphi\in C_0^\infty(B_1)$,
\begin{align}
        \int_{\mb{R}^n}v(-\Delta)^s\varphi \ dx-M_0\cdot\int_{B_1} v\nabla\varphi\ dx -\int_{B_1}q\nabla\varphi\ dx =0,\label{lineq}
\end{align}
    along with $\text{div }v=0$ (weakly), $(v)_{1}=0$, $(q)_{1}=0$, and $E(v,q;1)\leq 1$.
\end{Lem}

\begin{proof}
With the rescaling \eqref{rescale}, changing variables we compute
    \begin{align}
        (-\Delta)^{2s}u_k(r_kx)=\frac{\eps_k}{r_k^{2s}}(-\Delta)^{2s} v_k(x),\quad \nabla u_k(r_k x)=\frac{\eps_k}{r_k}\nabla v_k(x),\quad \nabla p_k(r_kx)=\frac{\eps_k}{r_k^{2s-1}}\nabla q_k(x).\notag
    \end{align}
    In addition we have
    \begin{align}
        \frac{1}{\eps_k}E^V(u_k;\theta r_k)&=E^V(v_k;\theta)=\left(\mean{B_\theta}|v_k-(v_k)_{\theta}|^3\ dx\right)^{\frac{1}{3}}\notag
        \\
        \frac{1}{\eps_k}E^P(p_k;\theta r_k)&=E^P(q_k;\theta)=\theta^{2s-1}\left(\mean{B_\theta}|p-(p)_{\theta}|^{\frac{3}{2}}\ dx\right)^{\frac{2}{3}}\notag
        \\
        \frac{1}{\eps_k}E^{nl}(u_k;\theta r_k)&=E^{nl}(v_k;\theta)=\left(\sup_{R\geq\frac{1}{4}\theta}\left(\frac{\theta}{R}\right)^{3s}\mean{B_R}|u-(u)_{\theta}|^2\ dx\right)^{\frac{1}{2}},\notag
    \end{align}
    so that 
	\begin{align}
		E(v_k,q_k;1)\leq 1\label{lpbound}
	\end{align}
We may therefore assume after passing to subsequences that there exists $(v,q)$ such that
    \begin{align}
        v_k\rightharpoonup v\quad\text{in}\quad L^3(B_1,\mb{R}^5),\quad q_k\rightharpoonup q\quad\text{in}\quad L^{\frac{3}{2}}(B_1,\mb{R}^5),\quad (v)_1=0,\quad (q)_1=0.\notag
    \end{align}

Now we discuss the compactness and strong convergence of $\{v_k\}$ and $\{q_k\}$. We wish to make use of the energy inequality \eqref{rescaledenineq} that we obtained in Lemma \ref{3.2}. Indeed the rescalings \eqref{rescale} are exactly of the form in \eqref{rdef} with $M=(u_k)_1$, $\tilde{p}=(p_k)_1$, and $L=\epsilon_k$. By \eqref{rescale} we have that $E^V(v_k;1)$, $E^{nl}(v_k;1)$, and $E^P(q_k;1)$ are bounded so that the following quantities are all bounded:
    \begin{align}
        \int_{B_1} |v_k|^2\ dx,\quad\sup_{R\geq 1} R^{-3s}\left(\mean{B_R}|u|\ dx\right)^2,\quad\int_{B_1} |v_k||q_k|\ dx.\notag
    \end{align}
Furthermore, the last term on the right-hand side of the energy inequality \eqref{rescaledenineq},
    \begin{align}
        \frac{r_k^{2s}}{\eps_k}\int_{B_1}|v_k||g_k|\ dx\leq\frac{r_k^{2s}}{\eps_k}\left(\mean{B_1}|g_k|\ dx^2\right)^{\frac{1}{2}}\left(\mean{B_1}|v_k|^2\ dx\right)^{\frac{1}{2}}\leq r_k^{\gamma-\beta}\left(\mean{B_1}|v_k|^2\ dx\right)^{\frac{1}{2}},\notag
    \end{align}
is also bounded, since $\gamma>\beta$. As a result, \eqref{rescaledenineq} combined with Lemma \ref{3.4} shows that independently of $k$,
    \begin{align}
        \int_{B_{\frac{3}{4}}^+} y^b|\oDelta_b v_k^*|^2\ dx\ dy\leq C.\notag
    \end{align}
Thus, using the proof of Lemma \ref{4.4} and the same test function $\varphi$ we find that $\|v_k\varphi\|_{\dot{H}^s(\mb{R}^5)}$ is bounded. Therefore by the fractional Rellich-Kondrachov theorem, since $\frac{5}{5-2s}>3$ when $s>1$ we can extract a subsequence of $\{v_k\}$ that converges strongly in $L^3(B_{1/2},\mb{R}^5)$, and the strong convergence of $q_k$ in $L^{\frac{3}{2}}(B_{1/2})$ is obtained by the Calder\'{o}n-Zygmund estimates (see \cite{L}). For the convergence in $\mb{R}^5$ in the sense of tempered distributions we use the fact that the nonlocal excess $E^{nl}(v_k;1)$ controls the $L^2(B_r)$ norm for any $r>0$, and for $p_k$ we use the fact that $\Delta q_k=\text{div}~\text{div}(v_k\otimes v_k)$.

By now we have shown the strong and weak convergence properties of $\{(v_k,q_k)\}$. Also by the lower semicontinuity of $E$ it follows from \eqref{lpbound} that $E(v,q;1)\leq 1$. Therefore it only remains to show that $(v,q)$ satisfies \eqref{lineq}.
    From the rescaling \eqref{rescale} we have that $(v_k,p_k)$ and $g_k$ satisfy, for all $\varphi\in C_0^\infty(B_1)$,
    \begin{align}
        &\int_{\mb{R}^n} v_k(-\Delta)^s\varphi\ dx-r_k^{2s-1}\eps_k\int_{B_1}(v_k\nabla\varphi) v_k\ dx-r_k^{2s-1}(u_k)_{r_k}\cdot\int_{B_1}v_k\nabla\varphi\ dx-\int_{B_1}q_k\nabla\varphi\ dx\notag
	\\
	&\quad=\frac{{r_k}^{2s}}{\eps_k}\int_{B_1}g_k\varphi\ dx.\notag
    \end{align}
Now observe that by hypothesis $\left|r_k^{2s-1}(u_k)_{r_k}\right|<M$ so we may assume $r_k^{2s-1}(u_k)_{r_k}\rightarrow M_0\in\mb{R}^n$ with $|M_0|\leq M$. Furthermore we may estimate the term involving $g_k$ by
    \begin{align}
        \frac{r_k^{2s}}{\eps_k}\left(\mean{B_1}|g_k|^2\ dx\right)^{\frac{1}{2}}&=\frac{r_k^{2s}}{\eps_k}\left(\mean{B_{R_k}}|f_k|^2\ dx\right)^{\frac{1}{2}}\leq \frac{d_k r_k^{\gamma-2s}}{\eps_k} r_k^{2s}=\frac{d_k r_k^{\beta}}{\eps_k} r_k^{\gamma-\beta}\leq r_k^{\gamma-\beta}\rightarrow 0.\notag
    \end{align}
    
    Therefore we find as desired that $v$ and $q$ satisfy for any $\varphi\in C_0^\infty(B_1)$,
    \begin{align}
        \int_{\mb{R}^n}v(-\Delta)^s\varphi\ dx-M_0\cdot\int_{B_1} v\nabla\varphi\ dx -\int_{B_1}q\nabla\varphi\ dx=0,\notag
    \end{align}
    along with $\text{div }v=0$ (weakly).
\end{proof}

\begin{Lem}\label{5.1}
    Let $s\in(1,2)$ and $M_0\in\mb{R}^5$ with $|M_0|\leq M$. Then there exists a $C>0$ depending on $M$ and $s$ such that if $v\in L^2_{\text{loc}}(\mb{R}^5)$ and $q\in L^{3/2}(B_1)$ satisfies \eqref{lineq}---namely that for all $\varphi\in C_0^\infty(B_1)$,
    \begin{align}
        \int_{\mb{R}^n}v(-\Delta)^s\varphi \ dx-M_0\cdot\int_{B_1} v\nabla\varphi\ dx -\int_{B_1}q\nabla\varphi\ dx =0,\notag
    \end{align}
    along with $\text{div }v=0$ (weakly), $(v)_{1}=0$, $(q)_{1}=0$, and $E^{nl}(v;1)<\infty$, then
    \begin{align}
        [v]_{C^1(B_{1/2})}+[q]_{C^1(B_{1/2})}\leq C\left(\int_{B_1}|v|^3\ dx\right)^{\frac{2}{3}}+C\left(\int_{B_1}|q|^{\frac{3}{2}}\ dx\right)^{\frac{4}{3}}+C\left(\sup_{R\geq\frac{1}{4}} R^{-3s}\mean{B_R}|v|^2\ dx\right)^{\frac{1}{2}}.
    \end{align}
\end{Lem}
\begin{proof}
    The system is linear with constant coefficients, so we can regularize and assume $v,q\in C^\infty$. First we multiply the equation by $v\varphi_1$ for $\varphi_1(x)=\Phi_1(x,0)$, where $\Phi_1(x,y)$ is a cutoff function between $B_{7/8}^+$ and $B_1^+$, so that
    \begin{align}
        0=\int_{B_1}\left(q\nabla\varphi_1\cdot v\varphi_1+M_0\cdot\nabla\varphi_1|v|^2\varphi_1-(-\Delta)^s v\cdot v\varphi_1^2\right)\ dx.\notag
    \end{align}
    Then we proceed as in the proof of Lemma \ref{3.2} to find that
    \begin{align}
        \int_{B_{3/4}^+}y^b|\oDelta_b v^*|^2\ dx\ dy&\leq C\left(\int_{B_1}|q|^{3/2}\ dx\right)^{2/3}\left(\int_{B_1}|v|^3\  dx\right)^{1/3}+CM\int_{B_1}|v|^2\ dx\notag
	\\
	&\quad+C\int_{B_1^+}y^b|v^*|^2\ dx\ dy.\notag
    \end{align}
    Now taking another cutoff function $\Phi_{3/4}$ between $B_{1/2}^+$ and $B_{3/4}^+$ and seek to control $\nabla u$ (we denote $\varphi_{3/4}=\Phi_{3/4}|_{y=0}$). First we have
    \begin{align}
        \int_{B_{3/4}}|\nabla v|^2\ dx&\leq\int_{\mb{R}^5}|\nabla(v\varphi_{3/4})|^2\ dx\leq C\int_{\mb{R}^5}|(-\Delta)^{s/2}(v\varphi_{3/4})|^2\ dx\leq C\int_{\mb{R}^6_+}y^b|\oDelta_b(v\varphi_{3/4})^*|^2\ dx\ dy,\notag
	\end{align}
	By Theorem \ref{Rays} the rightmost term is bounded by $C\int_{\mb{R}^6_+} y^b|\oDelta_b (v^*\Phi_{3/4})|^2\ dx\ dy$, and then by Lemma \ref{3.3} we have
	\begin{align}
        C\int_{\mb{R}^6_+}y^b|\oDelta_b(v^*\Phi_{3/4})^*|^2\ dx\ dy&\leq C\int_{\mb{R}^6_+}y^b\left(|\oDelta_b v^*|^2\Phi_{3/4}^2+|\onabla v^*|^2|\onabla\Phi_{3/4}|^2+|v^*|^2|\oDelta_b\Phi_{3/4}|^2\right)\ dx\ dy\notag
        \\
        &\leq C\int_{B_1^+}y^b|\oDelta_b v^*|^2\varphi_{3/4}^2\ dx\ dy+C\int_{B_1^+}y^b|v^*|^2\ dx\ dy.\notag
    \end{align}
    Next we establish higher order control of the pressure. Let $\psi\in C_0^\infty(B_{3/4})$ be a cutoff function which is identically $1$ on $B_{23/32}$, and let $\hat{q}$ solve $\Delta\hat{q}=\text{div }(M_0\cdot\nabla(\psi v))$. Here we are estimating on a sequence of concentric balls with decreasing radii, which as we will see below are taken as $3/4$, $23/32$, $11/16$, $5/8$, $9/16$. The explicit choice of radii, however, is not important, though the choice d affect the size of the constants $C$ in the estimates throughout.

By the Calder\'{o}n-Zygmund estimates we have
    \begin{align}
        \|\nabla\hat{q}\|_{L^2(B_{11/16})}&\leq\|M_0\cdot\nabla(\psi v)\|_{L^2(B_{11/16})}\leq C\|\nabla v\|_{L^2(B_{3/4})}+C\|v\|_{L^2(B_{3/4})}\notag
        \intertext{and since $\hat{q}-q$ is harmonic,}
        \|\nabla(\hat{q}-q)\|_{L^2(B_{11/16})}&\leq \|\hat{q}-q\|_{L^{3/2}(B_{23/32})}\leq C\left(\|q\|_{L^{3/2}(B_{3/4})}+\|v\|_{L^2(B_{3/4})}\right).\notag
    \end{align}
    Therefore,
    \begin{align}
        \|\nabla q\|_{L^2(B_{11/16})}\leq C\|\nabla v\|_{L^2(B_{3/4})}+C\|v\|_{L^2(B_1)}+C\|q\|_{L^{3/2}(B_1)}.\notag
    \end{align}
    Now we have controlled the $L^2$ norms of both $\nabla u$ and $\nabla p$. To control higher order derivatives, by the linearity of \eqref{lineq} we can differentiate by applying $\partial_i$ to obtain
    \begin{align}
        \int_{\mb{R}^n}\partial_i v(-\Delta)^s\varphi \ dx-M_0\cdot\int_{B_1} \partial_iv\cdot\nabla\varphi\ dx -\int_{B_1}\partial_i q\nabla\varphi\ dx =0.\notag
    \end{align}
    Take a cutoff function $\varphi_{11/16}$ between $B_{5/8}$ and $B_{11/16}$. Multiplying $\varphi_{11/16}$ against $\partial_i u$ and applying the same argument as before yields
    \begin{align}
        \int_{B_{5/8}^+}y^b|\oDelta_b \partial_i v^*|^2\Phi_{11/16}\ dx\ dy&\leq C\int_{B_{11/16}}\partial_iq\nabla\varphi_{11/16}\cdot\partial_i v\varphi_{11/16}\ dx+C\int_{B_{11/16}}|\partial_i v|^2\varphi_{11/16}\ dx\notag
	\\
	&\quad+C\int_{B^+_{11/16}}y^b|\partial_i v^*|^2\ dx\ dy.\notag
    \end{align}
    Here in analogy to the previous notation $\Phi_{11/16}(x,0)=\varphi_{11/16}(x)$ and $\Phi_{11/16}$ is a cutoff function between $B_{5/8}^+$ and $B_{11/16}^+$. The first term on the right is bounded by $C\|\nabla q\|_{L^2(B_{11/16})}^2+C\|\nabla v\|_{L^2(B_{11/16})}^2$, so we may conclude
    \begin{align}
        \int_{B^+_{5/8}}y^b|\oDelta_b\partial_i v^*|^2\leq C\|v\|_{L^3(B_1)}^2+C\|q\|^2_{L^{3/2}(B_1)}+C\int_{B_1^+}y^b|v^*|^2\ dx\ dy.\notag
    \end{align}
    Iterating these estimates $k$ times (on an appropriately chosen decreasing sequence of radii) then shows
    \begin{align}
        \|v\|_{H^k(B_{9/16})}\leq C\|v\|^2_{L^3(B_1)}+C\|q\|^2_{L^{3/2}(B_1)}+C\int_{B_1^+}y^b|v^*|^2\ dx\ dy,\notag
        \intertext{and for the pressure,}
        \|q\|_{H^k(B_{9/16})}\leq C\|v\|_{L^3(B_1)}+C\|q\|_{L^{3/2}(B_1)}+C\int_{B_1^+}y^b|v^*|^2\ dx\ dy.\notag
    \end{align}
    From Morrey's embedding theorem when $k=4$ combined with Lemma \ref{3.4} which helps bound the right-hand sides above we obtain the desired conclusion, since
    \begin{align}
        \|v\|_{C^1(B_{9/16})}\leq C\|v\|_{H^4(B_{9/16})},\quad \|q\|_{C^1(B_{9/16})}\leq C\|q\|_{H^4(B_{9/16})}.\notag
    \end{align}
\end{proof}

\begin{proof}[Proof of Proposition \ref{6.1}]

Recall that we start with a sequence $\{(u_k,p_k)\}$ of steady suitable weak solutions to \eqref{NS5} with force $f_k$ and $r_k\rightarrow 0$ satisfying \eqref{contr1}--\eqref{contra}, and aim to produce a contradiction with \eqref{contra} in order to prove the Proposition. In the proof of Lemma \ref{compact} we rescaled $(u_k,p_k)$ and $f_k$ to $(v_k,q_k)$ and $g_k$ in \eqref{rescale}. As a result \eqref{contr1} implies $E(v_k,q_k;1)\leq 1$, which is exactly the inequality \eqref{lpbound} that we used earlier, and \eqref{contra}, which we have not used until now, becomes equivalent to
\begin{align}
	E(v_k,q_k;\theta)\geq c_1\theta^{\alpha_1}.\label{realcontra}
\end{align}
Therefore to prove the proposition it now suffices to produce a contradiction with \eqref{realcontra}. We will show that for sufficiently large $k$,
\begin{align}
	E^V(v_k;\theta)\leq C\theta,\quad E^P(q_k;\theta)\leq C\theta^{2s},\quad E^{nl}(v_k;\theta)\leq C(\theta+\theta^{\frac{3s}{2}}).\label{willshow}
\end{align}
This will give the contradiction; from \eqref{realcontra} we would see that $c_1\theta^{\alpha_1}\leq C(\theta+\theta^{2s}+\theta+\theta^{\frac{3s}{2}})$, which cannot be true for a $c_1>0$ chosen sufficiently large.

First we look at $E^V$ and $E^P$. From Lemma \ref{compact} we know that $(v_k,q_k)$ converges to a pair $(v,q)$ which satisfies the equation \eqref{lineq}. Lemma \ref{5.1} gives estimates for such a solution $(v,q)$ and by the bound \eqref{lpbound} we therefore find that
\begin{align}
	E^V(v;\theta)\leq C\theta,\quad E^P(q;\theta)\leq C\theta^{2s}.\notag
\end{align}
Furthermore, because of the strong convergence $v_k\rightarrow v$ in $L^3(B_{\frac{1}{2}})$ and $q_k\rightarrow q$ in $L^{\frac{3}{2}}(B_{\frac{1}{2}})$ it follows that for $k$ sufficiently large,
\begin{align}
	E^V(v_k;\theta)\leq C\theta,\quad E^P(q_k;\theta)\leq C\theta^{2s}.\label{willshow1}
\end{align}
It remains to look at $E^{nl}$. We estimate
    \begin{align}
        E^{nl}(v_k;\theta)&=\left(\sup_{R\geq\frac{1}{4}\theta}\left(\frac{\theta}{R}\right)^{3s}\mean{B_R}|v_k-(v_k)_\theta|^2\ dx\right)^{\frac{1}{2}}\notag
        \\
        &\leq \left(\sup_{\frac{1}{4}\theta\leq R<\frac{1}{4}}\left(\frac{\theta}{R}\right)^{3s}\mean{B_R}|v_k-(v_k)_\theta|^2\ dx+\sup_{R\geq\frac{1}{4}}\left(\frac{\theta}{R}\right)^{3s}\mean{B_R}|v_k-(v_k)_\theta|^2\ dx\right)^{\frac{1}{2}}\notag
    \end{align}
    The second term on the right can be bounded as follows:
    \begin{align}
        \left(\sup_{R\geq\frac{1}{4}}\left(\frac{\theta}{R}\right)^{3s}\mean{B_R}|v_k-(v_k)_\theta|^2\ dx\right)^{\frac{1}{2}}&\leq\theta^{\frac{3s}{2}}E(v_k,q_k;1)+(4\theta)^{\frac{3s}{2}}|(v_k)_{\theta}-(v_k)_{1}|\notag
        \\
        &\leq \theta^{\frac{3s}{2}}+(4\theta)^{\frac{3s}{2}}\left(|(v_k)_{\theta}-(v_k)_{{1/2}}|+|(v_k)_{{1/2}}-(v_k)_{1}|\right)\notag
    \end{align}
    Noting that $|(v_k)_{B_{1/2}}-(v_k)_{B_1}|\leq C E^V(v_k;1)\leq C$ then leads to the bound on $E^{nl}$,
    \begin{align}
        E^{nl}(v_k;\theta)\leq \left(\sup_{\frac{1}{4}\theta\leq R<\frac{1}{4}}\left(\frac{\theta}{R}\right)^{3s}\mean{B_R}|v_k-(v_k)_\theta|^2\ dx\right)^{\frac{1}{2}}+\theta^{\frac{3s}{2}}+C(4\theta)^{\frac{3s}{2}}\left(1+|(v_k)_{\theta}-(v_k)_{{1/2}}|\right).\notag
    \end{align}
    Because of the bounds of \eqref{lpbound}, the estimates of Lemma \ref{5.1} imply $|(v)_{\theta}-(v)_{{1/2}}|\leq C$, and additionally for every $x\in B_R$ with $R\leq\frac{1}{4}$, $|v-(v)_{\theta}|\leq C R$. Hence
    \begin{align}
        \left(\sup_{\frac{1}{4}\theta\leq R<\frac{1}{4}}\left(\frac{\theta}{R}\right)^{3s}\mean{B_R}|v-(v)_\theta|^2\ dx\right)^{\frac{1}{2}}\leq \left(\sup_{\frac{1}{4}\theta\leq R<\frac{1}{4}}\left(\frac{\theta}{R}\right)^{3s}R^2\right)^{\frac{1}{2}}\leq C\theta.\notag
    \end{align}
    We can now apply the strong convergence $v_k\rightarrow v$ in $L^3(B_{\frac{1}{2}})$, $q_k\rightarrow q$ in $L^{3/2}(B_{\frac{1}{2}})$ to see that
    \begin{align}
        E^{nl}(v_k;\theta)\leq C(\theta+\theta^{\frac{3s}{2}}).\label{willshow2}
    \end{align}
    Thus \eqref{willshow1} and \eqref{willshow2} combine to give \eqref{willshow} and complete the proof.
\end{proof}

\subsection{Proof of Theorem \ref{main}}

Starting from Proposition \ref{6.1} we now proceed to prove Theorem \ref{main}. As described earlier we first iterate Proposition \ref{6.1} and then relate the smallness hypothesis on the excess $E$ to the smallness hypothesis of Theorem \ref{main} on the related energy $\ol{E}$.

\begin{Lem}[Iteration of Proposition \ref{6.1}]\label{2.5}
    Let $\theta, M,\beta,\beta_1$ be taken so that
    \begin{align}
        M\geq 3,\quad 0<\beta_1\leq\beta<\gamma,\quad 0<\beta_1<\alpha_1,\notag
        \\
        0<\theta\leq\frac{1}{2},\quad c_1(M)\theta^{\frac{\alpha_1-\beta_1}{2}}\leq 1.\notag
    \end{align}
    There exists $\ol{\eps_1}$ such that if $(u,p)$ are steady suitable weak solutions of \eqref{NS5} with force $f\in M_{2s,\gamma}$ and
    \begin{align}
        \begin{cases}
            A(u;R)=r^{2s-1}|(u)_{r}|<\frac{M}{2}
            \\
            E(u,p;r)+c_\gamma(f) r^\beta<\ol{\eps_1}
        \end{cases}\notag
    \end{align}
    for $r<R_1$, with $\eps_1$ and $R_1$ as in Proposition \ref{6.1}, then following inequalities hold for $k=0,1,2,\ldots$:
    \begin{align}
        \begin{cases}
            A(u;\theta^k r)=(\theta^k r)^{2s-1}|(u)_{{\theta^{k} r}}|<M
            \\
            E(u,p;\theta^k r)+c_\gamma(f)(\theta^k r)^\beta<\eps_1
            \\
            E(u,p;\theta^{k+1}r)\leq\theta^{(k+1)\beta_1}(1-\theta^{\frac{\alpha_1-\beta_1}{2}})^{-1}(E(u,p;r)+c_\gamma(f) r^\beta).
        \end{cases}
    \end{align}
\end{Lem}
\begin{proof}
    We prove the lemma by induction on $k$. When $k=0$ this is a consequence of Lemma \ref{5.1} if $\ol{\eps_1}<\eps_1$, since we have
    \begin{align}
        E(u,p;\theta r)&\leq c_1\theta^{\alpha_1}(E(u,p;r)+c_\gamma(f) r^\beta)\notag
        \\
        &\leq \theta^{\frac{\alpha_1+\beta_1}{2}}c_1\theta^{\frac{\alpha_1-\beta_1}{2}}(E(u,p;r)+c_\gamma(f) r^\beta)\notag
        \\
        &\leq \theta^{\frac{\alpha_1+\beta_1}{2}}(E(u,p;r)+c_\gamma(f) r^\beta).\notag
    \end{align}
    Now assume the statement is true for $\ell=0,\ldots,k$. We observe that
    \begin{align}
        |(u)_r-(u)_{\theta r}|&=\left|\mean{B_{\theta r}} u-(u)_r\ dx\right|\leq\left(\mean{B_{\theta r}} |u-(u)_r|^3\ dx\right)^{1/3}\leq\frac{1}{\theta^{\frac{5}{3}}}E^V(u;r).\notag
    \end{align}
    Thus,
    \begin{align}
        |(v)_{\theta^k r}|\leq \frac{1}{\theta^{\frac{5}{3}}}\sum_{i=0}^{k-1}E^V(u;\theta^i r)+|(v)_{r}|,\notag
    \end{align}
    and as a result,
    \begin{align}
        A(u;\theta^k r)\leq \frac{(\theta^k r)^{2s-1}}{\theta^{\frac{5}{3}}}\sum_{i=0}^{k-1}E^V(u;\theta^i r)+\theta^{k(2s-1)}A(u,r).\notag
    \end{align}
    Using this we may therefore compute that
    \begin{align}
        A(u;\theta^{k+1} r)&\leq \frac{(\theta^k r)^{2s-1}}{\theta^{\frac{5}{3}}}\sum_{i=0}^kE^V(u;\theta^i r)+\theta^{(k+1)(2s-1)}A(u;r)\notag
        \\
        &\leq\frac{R_1^{2s-1}}{\theta^{5/3}}\frac{1}{1-\theta^{\beta_1}}\frac{\ol{\eps_1}}{1-\theta^{\frac{\alpha_1-\beta_1}{2}}}+\frac{M}{2}\notag
        \\
        &< M,\notag
    \end{align}
    for $\ol{\eps_1}$ chosen sufficiently small (independent of $k$). To check the second condition we have that
    \begin{align}
        E(u,p;\theta^{k+1}r)+c_\gamma (\theta^{k+1} r)^\beta&\leq\frac{\theta^{(k+1)\beta_1}}{1-\theta^{\frac{\alpha_1-\beta_1}{2}}}(E(u,p;r)+c_\gamma(f) r^\beta)+c_\gamma(f)(\theta^{k+1} r)^\beta\notag
        \\
        &\leq\theta^{(k+1)\beta_1}\left(\frac{E(u,p;r)+c_\gamma(f) r^\beta}{1-\theta^{\frac{\alpha_1-\beta_1}{2}}}+c_\gamma(f) r^\beta\right)\notag
        \\
        &\leq\theta^{(k+1)\beta_1}2\frac{E(u,p;r)+c_\gamma(f) r^\beta}{1-\theta^{\frac{\alpha_1-\beta_1}{2}}}\notag
        \\
        &<\eps_1,\notag
    \end{align}
    again if $\ol{\eps_1}$ is sufficiently small, independent of $k$. Finally we check the last condition by applying Proposition \ref{6.1} to see that
    \begin{align}
        E(u,p;\theta^{k+2}r)&\leq\theta^{\frac{\alpha_1+\beta_1}{2}}(E(u,p;\theta^{k+1}r)+c_\gamma(f)(\theta^{k+1} r)^\beta)\notag
        \\
        &\leq \theta^{\frac{\alpha_1+\beta_1}{2}}\left(\theta^{(k+1)\beta_1}\frac{E(u,p;r)+c_\gamma(f) r^\beta}{1-\theta^{\frac{\alpha_1-\beta_1}{2}}}+c_\gamma(f)(\theta^{k+1} r)^\beta\right)\notag
        \\
        &\leq \theta^{\frac{\alpha_1+\beta_1}{2}}\theta^{(k+1)\beta_1}\left(\frac{E(u,p;r)+c_\gamma(f) r^\beta}{1-\theta^{\frac{\alpha_1-\beta_1}{2}}}+c_\gamma(f) r^\beta\right)\notag
        \\
        &=\theta^{(k+2)\beta_1}\theta^{\frac{\alpha_1-\beta_1}{2}}\left(\frac{E(u,p;r)+c_\gamma(f) r^\beta}{1-\theta^{\frac{\alpha_1-\beta_1}{2}}}+c_\gamma(f) r^\beta\right)\notag
        \\
        &\leq\theta^{(k+2)\beta_1}\frac{1}{1-\theta^{\frac{\alpha_1-\beta_1}{2}}}\left(E(u,p;r)+c_\gamma(f) r^\beta\right).\notag
    \end{align}
\end{proof}

Now we can relate a smallness hypothesis for $\ol{E}$ to the hypothesis on $E$ in Proposition \ref{6.1}.

\begin{Lem}\label{2.6}
    There exist numbers $\eps_0$ and $R_0$ such that if
    \begin{align}
        \ol{E}(u,p;r)+c_\gamma(f) r^{\frac{\gamma}{2}}<\eps_0,\label{2.6ineq}
    \end{align}
    for any $r<R_0$, then $u$ is H\"{o}lder continuous in some neighborhood of the origin.
\end{Lem}
\begin{proof}
    We take $R_0$ to be the constant $R_1$ obtained in Proposition \ref{6.1} when $\beta=\frac{\gamma}{2}$. Because $x\rightarrow\ol{E}(u,p;x,r)$ is continuous, if \eqref{2.6ineq} holds then there exists a neighborhood $\mc{O}$ of the origin such that
    \begin{align}
        \ol{E}(u,p;x,r)+c_\gamma(f)r^\beta<\eps_0,\quad\text{for all}\ x\in\mc{O}.\notag
    \end{align}
    Now we claim that $\ol{E}(u,p;x,r)+c_\gamma(f) r^\beta<\eps_0$ implies
\begin{subequations}
    \begin{align}
        A(u;x,r/2)&<\frac{3}{2},\label{tempcon1}
        \\
        E(u,p;x,r/2)+c_\gamma(f) (r/2)^\beta&<\ol{\eps_1}.\label{tempcon2}
    \end{align}
\end{subequations}
    Here $A(u;x,r)=r^{2s-1}|(u)_{B_r(x)}|$. For the first inequality we note that
    \begin{align}
        A(u;x,r)< R_0^{2s-1}\ol{E}(u,p;x,r)<\frac{3}{2}\notag
    \end{align}
    if $\eps_0$ is taken small enough. For the second inequality, we claim given $\ol{\eps_1}>0$ there exists $\eps_0$ such that if $\ol{E}(u,p;r)<\eps_0$ then $E(u,p;r/2)<\ol{\eps_1}$. For this it suffices by the scaling invariance of $r^{2s-1}\ol{E}(u,p;r)$ and $r^{2s-1}E(u,p;r)$ to show that there exists $\eps_0>0$ so that if $\ol{E}(u,p;2)<\eps_0$ then $E(u,p;1)<\ol{\eps_1}$; we observe that
    \begin{align}
        E^V(u;1)&\leq\left(\mean{B_1}|u|^3\ dx\right)^{\frac{1}{3}}+C|(u)_1|\leq 2\left(\mean{B_1}|u|^3\ dx\right)^{\frac{1}{3}}\leq C \eps_0,\notag
    \end{align}
    and similarly $E^P(u;1)\leq C\eps_0$. For the nonlocal part we have that $B_R(x)\subset B_{R+1}$ when $R\geq \frac{1}{4}$ so
    \begin{align}
        \left( \sup_{R\geq\frac{1}{4}}R^{-3s}\mean{B_R}|u-(u)_1|^2\ dx\right)^{\frac{1}{2}}\leq C\left(\sup_{R\geq\frac{1}{4}}R^{-3s}\mean{B_R}|u|^2\ dx\right)^{\frac{1}{2}}+C|(u)_1|\leq C\eps_0,\notag
    \end{align}
    since $|(u)_1|\leq C\|u\|_{L^3(B_1)}$.
    Finally it suffices to choose $\eps_0>0$ small enough so that
    \begin{align}
        C\eps_0+C\eps_0+C\eps_0<\ol{\eps_1}.\notag
    \end{align}

Having now justified \eqref{tempcon1} and \eqref{tempcon2} we may apply Lemma \ref{2.5} with $\beta_1=\frac{1}{2}\min\{\alpha_1,\gamma\}$, $M=3$, $\beta=\frac{\gamma}{2}$, and $\theta $ small enough to satisfy the hypotheses of Lemma \ref{2.5}. This implies that
    \begin{align}
        E(u,p;x,\rho)\leq C\left(\frac{\rho}{r/2}\right)^\beta_1,\notag
    \end{align}
    for all $\rho\in(0,r/2)$ and $x\in\mc{O}$, so that the Campanato criterion implies that $u$ is H\"{o}lder continuous in a neighborhood of the origin.
\end{proof}

We can now prove Theorem \ref{main}.

\begin{proof}[Proof of Theorem \ref{main}]
    It suffices to consider $x=0$. Let $R_0$ and $\eps_0$ be as in Lemma \ref{2.6}. Given $(u,p)$ and $f$ we choose $R>0$ so that
    \begin{align}
        R<\frac{R_0}{2},\quad\frac{2R}{R_0} c_\gamma(f) R^{\frac{\gamma}{2}}<\ol{\eps_0}.\notag
    \end{align}
    Now we use the scaling properties of \eqref{NS5}. If we define
    \begin{align}
        u_\tau(x)&=\tau^{2s-1}u(\tau x),\quad p_\tau(x)=\tau^{4s-2}p(\tau x),\quad f_\tau(x)=\tau^{4s-1}f(\tau x),\notag\label{scaling}
    \end{align}
    then $u_\tau$, $p_\tau$, and $f_\tau$ also satisfy \eqref{NS5} we have the scaling of the energy,
    \begin{align}
        \ol{E}(u_\tau,p_\tau;\frac{R_0}{2})=\tau^{2s-1}\ol{E}(u,p;\tau\frac{R_0}{2}),\notag
    \end{align}
    and the scaling of $c_\gamma$,
    \begin{align}
        c_\gamma(f_\tau)=\sup\left\{\frac{1}{r^{\gamma-2s}}\left(\mean{B_r(x)}|f_\tau|^2\ dx\right)^{\frac{1}{2}}\right\}.\notag
    \end{align}
    Observe also that
    \begin{align}
        \frac{1}{r^{\gamma-2s}}\left(\mean{B_r(x)}|f_\tau|^2\ dx\right)^{\frac{1}{2}}&=\frac{1}{r^{\gamma-2s}}\left(\mean{B_{\tau r}(x)}|\tau^{4s-1}f|^2\ dx\right)^{\frac{1}{2}}\notag
        \\
        &=\tau^{\gamma+4s-3}\frac{1}{(\tau r)^{\gamma-2s}}\left(\mean{B_{\tau r}(x)}|f|^2\ dx\right)^{\frac{1}{2}}\notag
        \\
        &\leq \tau^{\gamma+4s-3}c_\gamma(f).\notag
    \end{align}
    We may check therefore with $\tau=\frac{2R}{R_0}<1$ that
    \begin{align}
        \ol{E}(u_\tau,p_\tau;\frac{R_0}{2})+c_\gamma(f_\tau)(R_0/2)^{\gamma/2}&\leq\tau^{2s-1}\left(\ol{E}(u,p;R)+\tau^{\gamma+2s-2}c_\gamma(f)(R_0/2)^{\gamma/2}\right)\notag
        \\
        &\leq \tau^{2s-1}\left(\ol{E}(u,p;R)+\tau^{\gamma/2+2s-2}c_\gamma(f) R^{\gamma/2}\right)\notag
        \\
        &<C\ol{\eps_0}\notag
    \end{align}
    Therefore if $\ol{\eps_0}$ is chosen sufficiently small then by Lemma \ref{2.6}, $u_\tau$ and thus $u$ is H\"{o}lder continuous in a neighborhood of the origin.
\end{proof}

\section{Proof of Theorem \ref{epsmain}}\label{finalsection}

Theorem \ref{epsmain} will be proved in this section. It is an immediate consequence of Proposition \ref{tcddec} below, which in turn follows directly by combining Theorem \ref{main} with Lemmas \ref{BT}, \ref{C}, and \ref{D} which make up the remainder of this section. The arguments here are derived from those in \cite{CDM} and \cite{TY2} but the time-independence of \eqref{NS5} allows for some simplifications.

Before proceeding to the results of this section we need to define a higher order energy $\mc{E}$ and some quantities related to $\ol{E}$ in Section \ref{decaysection}.

We define the following energy which is the natural analogue in our case to the energy used in \cite{CDM},
\begin{align}
    \mc{E}(u;x,r)=r^{-7+6s}\int_{B_r^+(x)}y^b|\onabla(\nabla u)^\flat|^2\ dx\ dy,
\end{align}
and seek to show that the smallness of $\mc{E}$ will imply the hypothesis of Theorem \ref{main}. That is, we assume $\limsup_{r\rightarrow 0}\mc{E}(u;x,r)$ is small and seek to bound the three quantities in the hypothesis of Theorem \ref{main}. Note that $\mc{E}$ is scale invariant in the sense that $\mc{E}(u_\tau;x,r)=\mc{E}(u;x,\tau r)$ for the scaling $u\rightarrow u_\tau$ of \eqref{scaling}. In what follows it suffices to consider $x=0$ and $u$ is fixed, so we drop these from the arguments for $\mc{E}$ and simply write $\mc{E}(r)$ for this energy.

We further define the following quantities which are also scale-invariant:
\begin{alignat}{2}
	\quad \mc{T}(r)=r^{7s-2}\sup_{R\geq\frac{r}{4}}\frac{1}{R^{3s}}\mean{B_R}|u|^2\ dx,\quad\mc{C}(r)=r^{-8+6s} \int_{B_r}|u|^3\ dx,\quad \mc{D}(r)=r^{-8+6s}\int_{B_r}|p|^{\frac{3}{2}}\ dx.\notag
\end{alignat}

These quantities are closely related to the energy $\ol{E}$ of Theorem \ref{main}:
\begin{align}
    r^{2s-1}\ol{E}(u,p;r)=\mc{T}(r)^{1/2}+\mc{C}(r)^{1/3}+\mc{D}(r)^{2/3}.
\end{align}
Therefore to prove Theorem \ref{epsmain} it suffices to show Proposition \ref{tcddec} below, which is the main result of this section.

\begin{Prop}\label{tcddec}
    Given $\ol{\eps_0}>0$, there exists $\eps>0$ such that if $\limsup_{r\rightarrow 0}\mc{E}(r)<\eps$, then
    \begin{align}
        \limsup_{r\rightarrow 0}\mc{T}(r)+\mc{C}(r)+\mc{D}(r)<\ol{\eps_0}.
    \end{align}
\end{Prop}

The rest of this section is devoted to proving Lemmas \ref{BT}, \ref{C}, and \ref{D}, which together with Theorem \ref{main} directly imply Proposition \ref{tcddec}. In order to prove Lemma \ref{BT} though we first need to introduce and compare with $\mc{E}$ the following auxiliary scale-invariant quantity
\begin{align}
	    \mc{B}(r)=r^{-5+4s}\int_{B_r}|\nabla u(x,t)|^2\ dx.\notag
\end{align}
After proving this decay estimate for $\mc{B}$ in Lemma \ref{BB} we continue through the proofs of Lemmas \ref{BT}--\ref{D} to conclude this section.

\begin{Lem}\label{BB}
	Let $s\in(1,2)$. There exists a $C>0$ such that for any $u\in\dot{H}^s(\mb{R}^5)$,
	\begin{align}
		\limsup_{r\rightarrow 0}\mc{B}(r)\leq C\limsup_{r\rightarrow 0}\mc{E}(r).
	\end{align}
\end{Lem}
\begin{proof}
	We claim there exists $\theta\in(0,1/4)$ such that
    \begin{align}
        \mc{B}(\theta r)\leq\frac{1}{2}\mc{B}(r)+C\mc{E}(r).\label{Bclaim}
    \end{align}
    If so, then by the scaling invariance of $\mc{B}$ if we let $\limsup_{r\rightarrow 0}\mc{E}(r)=\eps$ we may take $r_0>0$ such that $\mc{E}(r)<2\eps$ for all $r<r_0$ and then apply \eqref{Bclaim} to see that
\begin{align}
	\mc{B}(\theta^k r)\leq\left(\frac{1}{2}\right)^k\mc{B}(r)+C\sum_{j=0}^{k-1}\left(\frac{1}{2}\right)^j\eps,\quad  k=1,2,\ldots\notag
\end{align}
which shows that there indeed exists $C$ such that $\limsup_{r\rightarrow 0}\mc{B}(r)\leq C\limsup_{r\rightarrow 0}\mc{E}(r)$. 

Again by the scaling invariance of $\mc{B}$ and $\mc{E}$, it suffices to prove this claim for $r=1$. First, we let $w=\nabla u$ so that $w^\flat=(\nabla u)^\flat$ (recall the definition of the extension $w^\flat$ from Theorem \ref{CaffS}). Therefore by the fundamental theorem of calculus,
\begin{subequations}
    \begin{align}
        w^\flat(x,y)&=w(x)+\int_0^y\partial_z w(x,z)\ dz\label{bequality1}
        \\
        w(x)&=w^\flat(x,y)-\int_0^y\partial_z w(x,z)\ dz.\label{bequality2}
    \end{align}
\end{subequations}
    Noting that $|\partial_z w^\flat|\leq|\onabla w^\flat|$, we can take the the square of \eqref{bequality1} and apply H\"{o}lder's inequality to obtain
	\begin{align}
		|w^\flat|(x,y)&\leq C|w|^2(x)+C\left(\int_0^y|\onabla w^\flat|\ dz\right)^2\notag
		\\
		&\leq C|w|^2(x)+C\left(\int_0^y z^b|\onabla w^\flat|^2\ dz\right)\left(\int_0^y z^{-b}\ dz\right)\notag
	\end{align}
	We can similarly do this for \eqref{bequality2}. Since $b=3-2s>-1$ we can therefore conclude that for $y\in[0,1]$,
\begin{subequations}
    \begin{align}
        |w^\flat|^2(x,y)\leq C |w|^2(x)+C\int_0^y z^b|\onabla w^\flat|^2(x,z)\ dz,\label{bint1}
        \\
        |w|^2(x)\leq C |w^\flat|^2(x,y)+C\int_0^y z^b|\onabla w^\flat|^2(x,z)\ dz.\label{bint2}
    \end{align}
\end{subequations}
    Integrating \eqref{bint1} on $B_1^+$ gives
    \begin{align}
        \int_{B_1^+}|w^\flat|^2\ dx\ dy\leq C\mc{B}(1)+C\mc{E}(1),\notag
    \end{align}
    and integrating the \eqref{bint2} on $B_\theta\times[\theta,2\theta]$ and then multiplying by $\theta^{-6+4s}$ gives
    \begin{align}
        \mc{B}(\theta)\leq C\theta^{-6+4s} \int_{B_\theta\times[\theta,2\theta]}|w^\flat|^2\ dx\ dy+C\theta^{-6+4s}\mc{E}(1).\notag
    \end{align}
    We will estimate above the first integral on the right:
    \begin{align}
        \int_{B_\theta\times[\theta,2\theta]}|w^\flat|^2\ dx\ dy&\leq C \int_{B_\theta\times[\theta,2\theta]}\left|w^\flat-\mean{B_1\times[\theta,1]}w^\flat\ d\tilde{x}\ d\tilde{y}\right|^2\ dx\ dy+C\theta^6\left|\mean{B_1\times[\theta,1]}w^\flat\ d\tilde{x}\ d\tilde{y}\right|^2\notag
        \\
        &\leq C\int_{B_1\times[\theta,1]}\left|w^\flat-\mean{B_1\times[\theta,1]}w^\flat\ d\tilde{x}\ d\tilde{y}\right|^2\ dx\ dy+C\theta^6\int_{B_1^+}|w^\flat|^2\ dx\ dy\notag
        \\
        &\leq C \int_{B_1\times[\theta,1]}|\onabla w^\flat|^2\ dx\ dy+C\theta^6\int_{B_1^+}|w^\flat|^2\ dx\ dy\notag
        \\
        &\leq C \theta^{-3+2s}\int_{B_1^+}y^b|\onabla w^\flat|^2\ dx\ dy+C\theta^6\int_{B_1^+}|w^\flat|^2\ dx\ dy\notag
        \\
        &\leq C \theta^{-3+2s}\mc{E}(1)+C\theta^6\int_{B_1^+}|w^\flat|^2\ dx\ dy.\notag
    \end{align}
    Putting everything together we therefore find that
    \begin{align}
        \mc{B}(\theta)&\leq C\theta^{4s} \int_{B_1^+}|w^\flat|^2\ dx\ dy+C(\theta^{-6+4s}+\theta^{-9+6s})\mc{E}(1)\notag
        \\
        &\leq C \theta^{4s}\mc{B}(1)+C(\theta^{4s}+\theta^{-6+4s}+\theta^{-9+6s})\mc{E}(1).\notag
    \end{align}
    Therefore choosing $\theta$ sufficiently small proves the claim and therefore concludes the proof.
\end{proof}

\begin{Lem}\label{BT}
    Let $s\in(1,2)$. There exists a $C>0$ such that for any function $u\in \dot{H}^s(\mb{R}^5)$,
    \begin{align}
        \limsup_{r\rightarrow 0}\mc{T}(r)\leq C\limsup_{r\rightarrow 0}\mc{E}(r).
    \end{align}
\end{Lem}
\begin{proof}  
    We claim there exists $\theta\in(0,1/4)$ such that
    \begin{align}
        \mc{T}(\theta r)\leq\frac{1}{2}\mc{T}(r)+C\mc{B}(r).\label{Tdec}
    \end{align}
    If so, then by applying Lemma \ref{BB} this would imply $\limsup_{r\rightarrow 0}\mc{T}(r)\leq C\limsup_{r\rightarrow 0}\mc{E}(r)$ by the same reasoning as given in the proof of Lemma \ref{BB}. Again it suffices to prove this for $r=1$ by scaling invariance. To this end we compute for $\rho\in[\frac{\theta}{4},\frac{1}{4}]$ that
    \begin{align}
        \mean{B_\rho}|u|^2\ dx&\leq C\mean{B_\rho}|u-(u)_{B_1}|^2\ dx+C\left|\mean{B_1}u\ dx\right|^2\notag
        \\
        &\leq C \frac{1}{\rho^5}\int_{B_\rho}|u-(u)_{B_1}|^2\ dx+C\sup_{R\geq\frac{1}{4}}\frac{1}{R^{3s}}\mean{B_R}|u|^2\ dx\notag
        \\
        &\leq C \frac{1}{\theta^5}\int_{B_1}|\nabla u|^2\ dx+C\sup_{R\geq\frac{1}{4}}\frac{1}{R^{3s}}\mean{B_R}|u|^2\ dx.\notag
        \intertext{Therefore}
        \sup_{R\geq\frac{\theta}{4}}\frac{1}{R^{3s}}\mean{B_R}|u|^2\ dx&\leq C\frac{1}{\theta^{3s}}\sup_{R\geq\frac{1}{4}}\frac{1}{R^{3s}}\mean{B_R}|u|^2\ dx+C\frac{1}{\theta^{5+3s}}\int_{B_1}|\nabla u|^2\ dx,\notag
    \end{align}
    and multiplying by $\theta^{7s-2}$ gives $\mc{T}(\theta)\leq C \theta^{4s-2}\mc{T}(1)+C\theta^{4s-7}\mc{B}(1)$,    so that we establish the claim and conclude the proof by taking $\theta$ sufficiently small.
\end{proof}

\begin{Lem}\label{C}
    Let $s\in(1,2)$. There exists a $C>0$ and $\eps>0$ such that for any function $u\in \dot{H}^s(\mb{R}^5)$, if $\limsup_{r\rightarrow 0}\mc{E}(r)<\eps$ then
    \begin{align}
        \limsup_{r\rightarrow 0}\mc{C}(r)\leq C\limsup_{r\rightarrow 0}\mc{E}(r).
    \end{align}
\end{Lem}
\begin{proof}
    We compute
    \begin{align}
        \int_{B_\theta}|u|^3\ dx&=\int_{B_\theta}|u|(|u|^2-|(u)_{B_1}|^2)\ dx+\int_{B_\theta}|u||(u)_1|^2\ dx\notag
        \\
        &\leq \left(\int_{B_1}|u|^3\ dx\right)^{1/3}\left(\int_{B_1}|u-(u)_{B_1}|^3\ dx\right)^{2/3}+|(u)_1|^2\left(\int_{B_\theta}|u|\ dx\right)\notag
        \\
        &\leq C \left(\int_{B_1}|u|^3\ dx\right)^{1/3}\left(\int_{B_1}|\nabla u|^2\ dx\right)+\theta^{10/3}\left(\int_{B_1}|u|^3\ dx\right)^{2/3}\left(\int_{B_\theta}|u|^3\ dx\right)^{1/3}\notag
        \\
        &\leq C \left(\int_{B_1}|u|^3\ dx\right)^{1/3}\left(\int_{B_1}|\nabla u|^2\ dx\right)+\frac{2}{3}\theta^5\int_{B_1}|u|^3\ dx+\frac{1}{3}\int_{B_\theta}|u|^3\ dx.\notag
    \end{align}
    Therefore we can absorb the last term on the right to obtain
    \begin{align}
        \mc{C}(\theta)&\leq C \theta^{6s-3}\mc{C}(1)+C\theta^{6s-8}\mc{C}(1)^{1/3}\mc{B}(1)\notag
        \\
        &\leq C\theta^{6s-3}\mc{C}(1)+C\theta^{3(6s-8)}\mc{C}(1)\mc{B}(1)+C\mc{B}(1).\notag
    \end{align}
    So if we choose $\eps$ sufficiently small and then $\theta$ sufficiently small we have
    \begin{align}
        \mc{C}(\theta)&\leq\frac{1}{2}\mc{C}(1)+C\mc{B}(1),
    \end{align}
    which implies the lemma.
\end{proof}

\begin{Lem}\label{D}
    Let $s\in(1,2)$. There exists a $C>0$ and $\epsilon>0$ such that for any suitable weak steady solution $(u,p)$ of \eqref{NS5}, if $\limsup_{r\rightarrow 0}\mc{E}(r)<\eps$ then
    \begin{align}
        \limsup_{r\rightarrow 0}\mc{D}(r)\leq C\limsup_{r\rightarrow 0}\mc{E}(r).
    \end{align}
\end{Lem}
\begin{proof}
    Let $\varphi$ be a cutoff function between $B_{1/2}$ and $B_{3/4}$ and denote the fundamental solution of $\Delta$ in $\mb{R}^5$ by $-C_5\frac{1}{|x|^3}$, where $C_5>0$. We will decompose $p$ as in \cite{CKN}. First,
    \begin{align}
        p(x)\varphi(x)&=-C_5\int_{\mb{R}^5}\frac{1}{|x-y|^3}\Delta_y(p\varphi)\ dy\notag
        \\
        &=-C_5\int_{\mb{R}^5}\frac{1}{|x-y|^3}[p\Delta\varphi+2\nabla p\cdot \nabla\varphi+\Delta p]\ dy.\notag
    \end{align}
    So we may write $p=p_1+p_{2,1}+p_{2,2}$, with
\begin{subequations}
    \begin{align}
        p_1(x)&=C_5\int_{\mb{R}^5}\frac{p(y)}{|x-y|^3}\Delta\varphi(y)\ dy+6C_5\int_{\mb{R}^5}\frac{p(y)}{|x-y|^5}(x-y)\cdot\nabla\varphi(y)\ dy,
        \\
        p_{2,1}(x)&=C_5\int_{\mb{R}^5}\frac{1}{|x-y|^3}(u^i-u^i_1)(u^j-u^j_1)\nabla_{ij}\varphi\ dy\notag
        \\
        &\quad+2C_5\int_{\mb{R}^5}\nabla_i\left(\frac{1}{|x-y|^3}\right)(u^i-u^i_1)(u^j-u^j_1)\nabla_j\varphi\ dy,
        \\
        p_{2,2}(x)&=C_5\int_{\mb{R}^5}\nabla_{ij}\left(\frac{1}{|x-y|^3}\right)(u^i-u^i_1)(u^j-u^j_1)\varphi\ dy
    \end{align}
\end{subequations}
    For $p_1$, from the choice of cutoff function we have
    \begin{align}
        |p_1(x)|\leq C\left(\mean{B_1}|p|^{3/2}\ dx\right)^{2/3},\quad\text{for all }x\in B_{1/4},\notag
    \end{align}
    and therefore
    \begin{align}
        \theta^{-8+6s}\int_{B_\theta}|p_1|^{3/2}\ dx\leq C\theta^{-3+6s}\mc{D}(1).\notag
    \end{align}
    Continuing for $p_{2,1}$ we obtain
    \begin{align}
        |p_{2,1}(x)|&\leq C \int_{B_1}|u-u_1|^2\ dx\leq C\int_{B_1}|\nabla u|^2\ dx,\notag
        \\
        \theta^{-8+6s}\int_{B_\theta}|p_{2,1}|^{3/2}\ dx&\leq C\theta^{-3+6s}\mc{B}(1).\notag
    \end{align}
    Finally for $p_{2,2}$ we have by the Calder\'{o}n-Zygmund theorem that
    \begin{align}
        \int_{B_\theta}|p_{2,2}|^{3/2}\ dx&\leq\int_{B_1}|u-u_1|^3\ dx\leq C\left(\int_{B_1}|\nabla u|^2\ dx\right)^{3/2},\notag
        \\
        \theta^{-8+6s}\int_{B_\theta}|p_{2,2}|^{3/2}\ dx&\leq C\theta^{-8+6s}\mc{B}(1)^{3/2}.\notag
    \end{align}
    Putting everything together we conclude that
    \begin{align}
        \mc{D}(\theta)\leq C\theta^{-3+6s}\mc{D}(1)+C\theta^{-3+6s}\mc{B}(1)+C\theta^{-8+6s}\mc{B}(1)^{3/2},\notag
    \end{align}
    so if we choose $\theta$ sufficiently small we obtain
    \begin{align}
        \mc{D}(\theta)\leq\frac{1}{2}\mc{D}(1)+C\mc{B}(1)+C\mc{B}(1)^{3/2},
    \end{align}
    which implies the result of the lemma.
\end{proof}

\section*{Appendix}
\renewcommand{\theequation}{A.\arabic{equation}}
\setcounter{equation}{0}

We include here a short proof of Theorem \ref{supermain} which follows as a consequence of Theorem \ref{epsmain} by the standard covering argument in \cite{CKN}.

\begin{proof}[Proof of Theorem \ref{supermain}]
	By Theorem \ref{epsmain} we have that if
	\begin{align}
		\limsup_{r\rightarrow 0} r^{-7+6s}\int_{B_r^+(x)}y^b|\onabla(\nabla u)^\flat|^2\ dx\ dy<\eps
	\end{align}
	then $u$ is H\"{o}lder continuous in a neighborhood of $x$.
	Given $\delta>0$, for each $x\in\mc{S}$ take a $B_{r_x}^+(x)$ with $r_x<\delta$ such that
	\begin{align}
		r_x^{-7+6s}\int_{B_{r_x}^+(x)}y^b|\onabla(\nabla u)^\flat|^2\ dx\ dy>\eps.\notag
	\end{align}
	Let $V_\delta=\bigcup_x B_{r_x}^+(x)$, a neighborhood of $\mc{S}$. By the Vitali covering lemma there is a disjoint countable subfamily $\{B_{r_i}^+(x_i)\}$, where $r_i=r_{x_i}$, such that  such that $V_\delta\subseteq\bigcup_i B_{r_i}^+(x_i)$, and 
	\begin{align}
		\sum_i (5r_i)^{7-6s}&\leq 5^{7-6s}\sum_i\eps^{-1}\int_{B_{r_i}^+(x_i)}y^b|\onabla(\nabla u)^\flat|^2\ dx\ dy\leq 5^{7-6s}\eps^{-1}\int_{V_\delta}y^b|\onabla(\nabla u)^\flat|^2\ dx\ dy.\label{endapp}
	\end{align}
	Furthermore $\sum_i (5r_i)^6\leq \sum_i (5r_i)^{6s-1} (5r_i)^{7-6s}\leq(5\delta)^{6s-1}\sum_i (5r_i)^{7-6s}$, which implies that the Lebesgue measure of $V_\delta$ is bounded by
	\begin{align}
		\eps^{-1}5^6\delta^{6s-1}\int_{V_\delta}y^b|\onabla(\nabla u)^\flat|^2\ dx\ dy\leq\eps^{-1}5^6\delta^{6s-1}\int_{\mb{R}^{n+1}_+}y^b|\onabla(\nabla u)^\flat|^2\ dx\ dy.\notag
	\end{align}
	This implies that given $\tilde{\delta}>0$ there exists $\delta>0$ such that $\left|V_\delta\right|<\tilde{\delta}$. Therefore it follows from \eqref{endapp} that $\mc{H}^{7-6s}(\mc{S})=0$.
\end{proof}

\begin{Rem}
	We can observe that $y^b|\onabla(\nabla u)^\flat|^2$ is integrable by looking at its Fourier transform:
	\begin{align}
		\int_{\mb{R}^6_+} y^b|\onabla(\nabla u)^\flat|^2\ dx\ dy&\leq C\int_{\mb{R}^6_+}y^b\left(|\xi|^4|\hat{u}(\xi)\psi(|\xi|y)|^2+|\partial_y\hat{u}(\xi)\psi(|\xi| y)|^2\right)\ d\xi\ dy\notag
		\\
		&=C\int_{\mb{R}^6_+}y^b|\xi|^{4-1-b}|\hat{u}(\xi)|^2\left(|\psi(y)|^2+|\psi'(y)|^2\right)\ d\xi\ dy\notag
		\\
		&=C\int_{\mb{R}^5}|\xi|^{2s}|\hat{u}(\xi)|^2\ d\xi\notag
		\\
		&=C\|u\|_{\dot{H}^s}<\infty.
	\end{align}
\end{Rem}

\section*{Acknowledgments}
I would like to thank my advisor, Sun-Yung Alice Chang, for many helpful discussions and comments as well as her constant support and encouragement. I would also like to thank Maria Colombo and Huy Q. Nguyen for several helpful conversations. This material is based upon work supported by the National Science Foundation Graduate Research Fellowship under Grant No. DGE-1656466.


\end{document}